\documentclass[11pt]{article}
\usepackage{graphicx} 
\usepackage[margin=1.15in]{geometry}
\usepackage{caption}
\usepackage{subcaption}

\usepackage{amsfonts}
\usepackage{tabularray}
\usepackage{amsmath}
\usepackage{mathabx}
\usepackage[T1]{fontenc}

\usepackage{setspace}
\usepackage{bm}
\usepackage{authblk}

\usepackage{hyperref}       
\hypersetup{
  colorlinks   = true, 
  urlcolor     = blue, 
  linkcolor    = blue, 
  citecolor   = blue 
}
\usepackage{cleveref}

\usepackage[english]{babel}
\usepackage[nottoc]{tocbibind}
\usepackage{amsthm}
\usepackage[thinlines]{easytable}

\providecommand{\keywords}[1]
{
  \small	
  \textbf{Keywords: } #1
}

\newtheorem{theorem}{Theorem}
\newtheorem{corollary}{Corollary}
\newtheorem{lemma}{Lemma}
\newtheorem*{remark}{Remark}
\newtheorem{prop}{Proposition}

\theoremstyle{definition}
\newtheorem{definition}{Definition}

\theoremstyle{definition}

\DeclareMathOperator{\trace}{tr}
\DeclareMathOperator{\kernel}{ker}
\DeclareMathOperator{\spn}{span}
\DeclareMathOperator{\supp}{supp}

\usepackage[toc,page,title]{appendix}

\usepackage{sectsty}
\sectionfont{\centering}

\usepackage{tikz}
\usetikzlibrary{calc}
\usetikzlibrary{perspective}
\usetikzlibrary{decorations.markings}

\title{\bf{Characterizing the Lovasz theta function via walk generating functions}}

\vspace{2cm}

\author{Lasse Harboe Wolff}

\affil{\small{\emph{Department of Mathematical Sciences, University of Copenhagen}, \\  \emph{Universitetsparken 5, 2100 Denmark}} \\
E-mail: lhw@math.ku.dk}

\date{\small{\today} }

\begin{document}

\maketitle

\begin{abstract}
\begin{spacing}{1.05}
\noindent A new characterization of the Lovasz theta function is provided by relating it to the (weighted) walk-generating function, thus establishing a relationship between two seemingly quite distinct concepts in algebraic graph theory. An application of this new characterization is given by showing how it straightforwardly entails multiple natural generalizations of the Hoffman upper bound (on both the independence number and Lovasz number) to arbitrary non-regular graphs. These new bounds possess properties that make them advantageous to previously derived such generalizations. It will also be shown that the Lovasz theta function equals a natural relaxation of the independence number, here dubbed the spherical independence number -- the determination of which involves producing a vector corresponding to a generalized maximum independent set which might be significant for the maximum independent set problem. Lastly, the derivation of the new characterization involves proving a certain analysis result which may in itself be of interest. 
\end{spacing}
\end{abstract}



\section{Introduction}

The Lovasz theta function $\vartheta$, also called the Lovasz number, is an important and much studied graph invariant for multiple reasons. $\vartheta(G)$ was notably introduced by Lovasz \cite{lovasz_shannon_1979} as an upper bound on the (zero error) Shannon capacity $\Theta (G)$ of graph $G$ (see \eqref{eq:original def of lovasz} for the original definition of $\vartheta$). The Shannon capacity 
$$\Theta(G) := \sup_{k \in \mathbb{N}} \left( \alpha \left(  \boxtimes^k  G \right) \right)^{1/k}$$ was introduced by Shannon \cite{shannon_zero_1956} and equals the maximum size of an effective alphabet for a communication channel with confusability graph $G$ and is very difficult to determine in general, even for some of the simplest non-trivial graphs \cite{Alon_2005JumpsGraphPowe, polak_new_2019, Bohman_2005cycles}. $\alpha$ denotes the independence number and $\boxtimes$ the \emph{strong graph product} (see Section \ref{sec:preliminaries}). $\vartheta(G)$ is also an important function due to it satisfying the follows "sandwich inequality" \cite{lovasz_shannon_1979, Lovsz1986AlgorithmicTO}
\begin{align*}
\omega(G) \leq \vartheta(\overline{G}) \leq \chi (G) \ \: ,
\end{align*}
$\overline{G}$ being the complement of the graph $G$, and $\omega (G)$ and $\chi (G)$ are respectively the clique and chromatic number of $G$ -- two numbers of great importance and with many applications in graph theory and computer science, but which are both NP-hard to compute (and approximate) \cite{Hstad1999CliqueIH, Kilian_chromatic}. $\vartheta(G)$ itself can be approximated in polynomial time \cite{Grtschel1981TheEM}. Other reasons for the interest in studying the Lovasz theta function include the fact that graph theorists have derived an increasingly long and still expanding list of different characterizations which all, perhaps quite surprisingly, turn out to equal $\vartheta$ \cite{Knuth1994, Karger_1998, Galtman2000, Luz2005}, thus establishing some form on universality of graph optimization problems. $\vartheta$ even finds applications within the field of quantum foundations \cite{Winter_2014quantumCorrelations, Howard2014}. 
\\
\\
The main contribution of this paper is to derive two new characterizations of the Lovasz theta function $\vartheta$. One of them showing $\vartheta (G)$ to equal the \emph{spherical independence number} $\alpha_{\mathcal{S}}(G)$, see Definition \ref{def:spherical ind number} from Section \ref{sec:preliminaries}, which essentially equals a relaxation of the ordinary independence number $\alpha$ wherein Boolean indicator vectors are relaxed to elements of a unique sphere of real vectors. The other characterization uses the \emph{weighted walk-generating functions} $W_A (x)$ for $G$, see Definition \ref{def:weighted walk-generating function} from Section \ref{sec:preliminaries}, and it shows $\vartheta(G)$ to equal $W_A (x)$ when $x$ is minimized over the interval $[\lambda_{min}(A)^{-1}, \lambda_{max}(A)^{-1}]$ and $A$ is minimized over all real symmetric weighted adjacency matrices for $G$. Both of these new characterizations are the content of Theorem \ref{theorem:char lovasz numb} from Section \ref{sec:main results}, and it is proven in Section \ref{sec:proof of main results}.

This new result has several consequences. First of all, it establishes a new connection between two quite distinct areas of algebraic graph theory -- namely the Lovasz theta function and walk-generating functions. Secondly, the newly established relationship between $\vartheta(G)$ and walk-generating functions entails several natural generalizations of the Hoffman bound for non-regular graphs. These new bounds not only upper bound $\alpha(G)$ but also $\vartheta (G)$, and it is shown that at least one of them is guaranteed to always be better than or equal to a famous generalization of the Hoffman bound which was recently shown to also bound $\vartheta$. A slight elaboration of this application is given below. Thirdly, it follows from the new results that there will for any $n$ vertex graph $G$ always exist a real vector $\boldsymbol{v} \in \mathbb{R}^n$ satisfying $|\boldsymbol{v}|^2 = \vartheta(G)$, $|\boldsymbol{v}|^2 = \langle \boldsymbol{1}, \boldsymbol{v} \rangle$ and $\langle \boldsymbol{v}, A \boldsymbol{v} \rangle = 0$, and $\boldsymbol{v}$ is further the largest such vector satisfying the latter two conditions. $\boldsymbol{1}$ and $A$ are respectively the all ones vector and some weighted adjacency matrix for $G$. This $\boldsymbol{v}$ is straightforward to determine once $\vartheta(G)$ has been calculated, and it could serve as a first approximation to the $\{0,1\}^n$ vector indicating the maximum independent set of $G$. It can be hoped that this could help attacking the NP-hard maximum independent set problem, since $\boldsymbol{v}$ is in this context easy to calculate.

Finally, it should be noted that to prove the main result of this paper, a duality result regarding the location of the critical points of certain real functions $f(x)$ that can be written as a finite sum of reciprocal functions as $f(x)=\sum_{i=1}^N \frac{a_i}{1-b_i x}, \ a_i \geq 0$ had to be established. Even if this is of limited further use within graph theory, the result, which is the content of Theorem \ref{lemma:app:main lemma for reciprocal functs} from Appendix \ref{app:special reciprocal function}, might be an interesting result in and of itself.
\\
\\
We shall now say a few more words regarding the application of the new characterization of $\vartheta (G)$ to generalizations of the Hoffman bound. The Hoffman bound, first derived by Hoffman (although not published -- see e.g. \cite{haemers_hoffmans_2021} for a brief historical overview), states that if $G$ is an $n$-vertex regular graph then the independence number $\alpha(G)$ is upper bounded by
\begin{align} \label{eq:hoffman upper bound independence}
\alpha(G) \leq \frac{- \lambda_n n}{\lambda_1-\lambda_n} \: \ ,
\end{align}
where $\lambda_1$ and $\lambda_n$ are respectively the largest and smallest adjacency eigenvalues of $G$. The inequality \eqref{eq:hoffman upper bound independence} was later extended by Lovasz \cite{lovasz_shannon_1979} to also hold for the Lovasz number, i.e. if $G$ is an $n$-vertex regular graph, then 
\begin{align*}
\vartheta(G) \leq \frac{- \lambda_n n}{\lambda_1-\lambda_n} \: \ .
\end{align*}
Trying to generalize the Hoffman bound and extending the bounds to other quantities has long been an active area of research \cite{HAEMERS1978445, GODSIL2008721, Wocjan2018MoreTO, haemers_hoffmans_2021}. It is however not entirely obvious what should count as a generalization of the Hoffman bound. In this paper, was call a bound a \emph{natural} generalizations of the Hoffman bound if it upper bounds $\alpha(G)$ or $\vartheta(G)$ -- depending on which upper bound one seeks to generalize, it equals the Hoffman bound for regular graph and it can be expressed as a sufficiently simple function of the eigenvalues and eigenvectors of the adjacency matrix, Laplacian or some other simple graph matrix. Obviously, generalizations of the Hoffman bound on $\vartheta(G)$ are more powerful, since we always have $\alpha(G) \leq \vartheta(G)$. One of the most famous generalizations of Hoffman's bound, which is also natural in the above sense, is the following 
\begin{align} \label{eq:godsil bound}
\alpha(G) \leq n \left(1 - \frac{\delta}{\mu_1} \right) \: \ ,
\end{align}
which holds for all graphs \cite{vanDam1998, GODSIL2008721} and where $\mu_1$ and $\delta$ are respectively the largest Laplacian eigenvalue and minimum degree of $G$. It was recently proven, in the context of simplicial complexes \cite[section 5]{Bachoc2019}, that \eqref{eq:godsil bound} holds with $\alpha$ replaced by $\vartheta$, i.e. we get the following generalization of the Hoffman bound on $\vartheta$
\begin{align} \label{eq:bachoc bound theta}
\vartheta(G) \leq n \left(1 - \frac{\delta}{\mu_1} \right) \: \ .
\end{align}
In this paper, it shall be proven as a straightforward corollary of the main result -- Theorem \ref{theorem:char lovasz numb}, that the following natural generalization of the Hoffman bound on $\vartheta$ holds 
\begin{align} \label{eq:first instance walk gen bound}
\vartheta(G) \leq \min_{\lambda_n^{-1} \leq x \leq 0} W(x) \: \: ,
\end{align}
where $\lambda_n$ is the smallest adjacency eigenvalue and $W(x)$ is the (un-weighted) walk-generating function of $G$, which is a strictly convex function in the relevant interval and is easy to express in terms of the adjacency eigenvalues and eigenvectors. This is the content of Corollary \ref{cor:precise compl generaliz} from Section \ref{sec:main results}, which in turn implies several other even simpler, albeit less precise, natural generalizations of the Hoffman bound on $\vartheta$ -- see the discussion of Section \ref{sec:main results} and Corollary \ref{cor:imprecise simpl generalization}.

It turns out, as proven by Proposition \ref{prop:better Hbound}, that the new bound \eqref{eq:first instance walk gen bound} is guaranteed to be better or equal to the earlier mentioned generalization of the Hoffman bound. That is, for all graph, we have $ \min_{\lambda_n^{-1} \leq x \leq 0} W(x) \leq n \left(1 - \frac{\delta}{\mu_1} \right)$. This then also serves as an independent proof of the fact that $n \left(1 - \frac{\delta}{\mu_1} \right)$ upper bounds $\vartheta(G)$ as stated in \eqref{eq:bachoc bound theta}. The new bounds proved here also display other advantages over existing generalizations of the Hoffman bound. Many such existing generalizations are very sensitive to changes in the minimum degree $\delta$ to the extent that if $\delta=0$, they reduce to the trivial bound $\alpha(G) \leq n$ or $\vartheta(G) \leq n$. This is the case for both the bound mentioned above in
\eqref{eq:godsil bound} and \eqref{eq:bachoc bound theta}, but also holds for other well-known generalizations of the Hoffman bound like the one proven in \cite{HAEMERS1978445}. The new bounds proven here are not subject to this sensitivity -- adding an edgeless vertex to $G$ only increases the function $\min_{\lambda_n^{-1} \leq x \leq 0} W(x)$ by $1$.

\section{Preliminaries} \label{sec:preliminaries}

Before presenting the main results of this paper in Section \ref{sec:main results}, we shall first introduce some notation and background results from linear algebra and graph theory, including some important formulas for $\vartheta(G)$. We shall also introduce the concepts of a \emph{weighted walk-generating function} and the \emph{spherical independence number} of a graph, in definitions \ref{def:weighted walk-generating function} and \ref{def:spherical ind number} below, which will become important later.  
\\
\\
Vectors will always be denoted in lower-case boldface, like $\boldsymbol{v}$, and will here be elements of $\mathbb{R}^n$ or $\mathbb{C}^n$ -- most often of $\mathbb{R}^n$. We let $\langle \cdot , \cdot \rangle$ denote the relevant inner product, thus if $\left\{ \boldsymbol{e}_i \right\}_{i=1}^n$ is an ortho-normal basis for $\mathbb{C}^n$, i.e. $\langle \boldsymbol{e}_i , \boldsymbol{e}_j \rangle = \delta_{i j}$, and if $\boldsymbol{v}=\sum_{i=1}^n v_i \boldsymbol{e}_i$ and $\boldsymbol{u}=\sum_{i=1}^n u_i \boldsymbol{e}_i$ for some $v_i, u_i \in \mathbb{C}$, we have $\langle \boldsymbol{v} , \boldsymbol{u} \rangle = \sum_{i=1}^n \bar{v}_i u_i$, where $\bar{\cdot}$ denotes the complex conjugate. We also write $|\boldsymbol{v}|=\sqrt{\langle \boldsymbol{v}, \boldsymbol{v} \rangle}$ for the induced norm. 

Matrices $M$ will always be denoted as upper-case letters. $I$ will always denote the identity matrix, $J$ will always denote the matrix of all ones and $\boldsymbol{1}$ will denote the vector of all ones, i.e.
\begin{align*} 
\boldsymbol{1}= \begin{pmatrix}
    1 \\
    \vdots \\
    1 \\
\end{pmatrix} 
\ , \qquad \qquad 
J= \begin{pmatrix}
    1 & \hdots & 1\\
    \vdots & \ddots & \vdots \\
    1 & \hdots & 1\\
\end{pmatrix} \ \: .
\end{align*}
Note that $J=\boldsymbol{1} \boldsymbol{1}^\dagger$ (with $\boldsymbol{v}^\dagger$ here a shorthand for $\langle \boldsymbol{v} , \cdot \rangle$). If $A$ is a linear operator, we denote the spectrum of $A$ as $\sigma(A)$. Thus, by the spectral theorem, any Hermitian matrix $A$ can be written as $A = \sum_{i=1}^n \lambda_i \boldsymbol{u}_i \boldsymbol{u}_i^\dagger = \sum_{\lambda \in \sigma(A)} \lambda \ P_\lambda$, for some orthonormal basis $\{ \boldsymbol{u}_i \}_{i=1}^n$ and where $P_\lambda$ denotes the projection onto the $\lambda$ eigen-space of $A$. We shall denote the eigen-space of a linear operator $A$ associated with eigenvalue $\lambda$ as $E_\lambda (A)$. Thus, $E_\lambda(A)= \spn \{ \boldsymbol{u}_i  \mid 1 \leq i \leq n , \ \lambda_i=\lambda \}$, and further $P_\lambda = \sum_{\boldsymbol{u}_i \in E_\lambda(A)} \boldsymbol{u}_i \boldsymbol{u}_i^\dagger$. When $A$ is Hermitian such that its eigenvalues are real, we shall lastly denote $\lambda_{max}(A)$ and $\lambda_{min}(A)$ as respectively the largest and smallest eigenvalue of $A$, in which case we also abbreviate $E_{\lambda_{max}(A)}(A)$ and $E_{\lambda_{min}(A)}(A)$ simply as $E_{max}(A)$ and $E_{min}(A)$ respectively.
\\
\\
We now introduce some relevant graph-theoretic concepts. $G$ will here always be a simple graph, i.e. an un-directed graph with no self-loops (no vertices are connected to themselves). However, all following results can in fact easily be extended to graphs containing self-loops. $G$ has vertex set $V(G)$ and edge set $E(G)$. The \emph{order} of a graph $G$ is the number of vertices in the graph, and will be denoted by $n$, i.e. $|V(G)|=n$. For $i,j \in V(G)$, we write $i \sim j$ if vertices $i$ and $j$ are connected/adjacent in $G$ and $i \not\sim j$ otherwise. Many of the results of this paper are related to the \emph{adjacency matrix} of a graph $G$ which has components $A_{ij}=1$ if $i \sim j$ and $A_{ij}=0$ if $i \not\sim j$, for some fixed labeling of the vertices in $V(G)$. We shall also study the more general \emph{weighted adjacency matrices} of some graph $G$, which without further specification will denote the set of matrices that satisfying just $A_{ij}=0$ if $i \not\sim j$ in $G$. The weighted adjacency matrices considered here will however all be real and symmetric. The \emph{Laplacian} matrix $L$ of $G$ is defined as $L=D-A$ where $A$ is the ordinary adjacency matrix and $D$ is the \emph{degree matrix} of $G$, whose components are $D_{ij} = \delta_{ij} d_i$ with $d_i$ being the degree or number of vertices adjacent to vertex $i$ in $G$. 

As mentioned in the Introduction, we here let $G \boxtimes H$ denote the strong graph product between graphs $G$ and $H$, i.e. a graph with vertex set $V(G \boxtimes H) = V(G) \times V(H)$, and where $(i,j) \sim_{G \boxtimes H} (i',j')$ iff either $i \sim_G i'$ and $j \sim_H j'$, or $i \sim_G i'$ and $j = j'$, or $i=i'$ and $j \sim_H j'$. It is not difficult to show that the ordinary (un-weighted) adjacency matrix $A(G \boxtimes H)$ for $G \boxtimes H$ can be expressed in terms of the individual (un-weighted) adjacency matrices $A(G)$ and $A(H)$ for $G$ and $H$ as $A(G \boxtimes H) = (A(G)+I_G) \otimes (A(H) + I_H) - I_G \otimes I_H$ in the product basis. $I_G$ and $I_H$ are the identity matrices on the individual spaces, and $\otimes$ is the Kronecker product.

We now provide precise definitions of $\vartheta(G)$. It was already clear in \cite{lovasz_shannon_1979} that $\vartheta(G)$ had many different and quite distinct characterizations, but $\vartheta(G)$ was initially defined as
\begin{align} \label{eq:original def of lovasz}
\vartheta(G) = \min_{U, \boldsymbol{c}} \max_i \frac{1}{\langle \boldsymbol{c}, \boldsymbol{u}_i \rangle^2}
\end{align}
where the minimum is taken over all \emph{ortho-normal representations} $U$ of $G$, meaning sets of real unit-vectors $U=\left\{ \boldsymbol{u}_i \in \mathbb{R}^N \mid 1 \leq i \leq n, \ |\boldsymbol{u}_i|^2=1 \right\}$, where the vertices of $G$ are labeled by $i \in \{1,2,...,n \}$, such that $\langle \boldsymbol{u}_i , \boldsymbol{u}_j \rangle =0$ if vertices $i$ and $j$ are not connected and $\boldsymbol{c} \in \mathbb{R}^N$ for some $N$ (not necessarily equaling $n$). However, for the purposes of this paper, we shall instead of using \eqref{eq:original def of lovasz} mainly rely on the fact that $\vartheta(G)$ has been shown \cite{lovasz_shannon_1979} to equal
\begin{align} \label{eq:lovasz number max eigen def}
\vartheta(G) = \inf_B \lambda_{max}(B) \ , 
\end{align}
where $B$ ranges over all real symmetric matrices satisfying $B_{ij}=1$ whenever $i \not\sim j$ in $G$. 
\\
\\
We shall now introduce some less common concepts that will be important to the presentation of the main results of this paper in Section \ref{sec:main results}.

\begin{definition}[weighted walk-generating function] \label{def:weighted walk-generating function}
For a graph $G$ of order $n$ with a weighted Hermitian adjacency matrix $A = \sum_{\lambda \in \sigma(A)} \lambda \ P_\lambda = \sum_{i=1}^n \lambda_i \boldsymbol{u}_i \boldsymbol{u}_i^\dagger$, the resulting weighted walk-generating function w.r.t. $A$ is a real function $W_A (x)$ defined as
\begin{align*} 
W_A (x) = \langle \boldsymbol{1},  (I-xA)^{-1}  \boldsymbol{1} \rangle = \sum_{i=1}^n \frac{ \left| \langle \boldsymbol{1},  \boldsymbol{u}_i \rangle \right|^2 }{1-x \lambda_i} = \sum_{\lambda \in \sigma(A)} \frac{ \langle \boldsymbol{1}, P_\lambda \boldsymbol{1} \rangle }{1-\lambda x}
\end{align*}
\end{definition}
If $\langle \boldsymbol{1}, P_\lambda \boldsymbol{1} \rangle =0$ for some $\lambda \in \sigma(A)$, we take $W_A(x)$ to equal $\sum_{\mu \in \sigma(A) \backslash \lambda} \frac{ \langle \boldsymbol{1}, P_\mu \boldsymbol{1} \rangle }{1-\mu x}$. Hence, $W_A(\lambda^{-1})$ is then well-defined, even if $(I-A / \lambda)^{-1}$ itself is not. In case $A$ is taken to be simply the ordinary/un-weighted adjacency matrix, we say that $W_A(x)$ is the ordinary/un-weighted walk-generating function for $G$, and we write it simply as $W(x)$, when it is clear which graph it is associated to. 
\begin{remark}
In other works, it is often the matrix function $(I-x A)^{-1}$ itself which is called the walk-generating function. If this distinction had been of importance here, we might have called $W_A (x)$ a \emph{summed} or a \emph{total} walk-generating function. It can be shown that $W_A(x)$ is a generating function for the sum of walks in $G$ of fixed length, weighted by products the coefficients from $A$ -- thus motivating its name. 
\end{remark}

We next define the \emph{spherical independence number} of a graph. To do that, we first introduce the following sphere $\mathcal{S}_n$ in $\mathbb{R}^n$
\begin{align} \label{eq:def of real sphere}
\mathcal{S}_n := \left\{ \boldsymbol{v} \in \mathbb{R}^n \ \middle| \ \left|\boldsymbol{v}-\frac{1}{2} \boldsymbol{1} \right|^2=\frac{n}{4} \right\} = \left\{ \boldsymbol{v} \in \mathbb{R}^n \ \middle| \ \langle \boldsymbol{1} , \boldsymbol{v} \rangle =|\boldsymbol{v}|^2 \vphantom{1^{1^1}} \right\} \ \: .
\end{align}

\begin{definition}[Spherical independence number] \label{def:spherical ind number} Let $G$ be a graph of order $n$. Then, the spherical independence number $\alpha_\mathcal{S} (G)$ of $G$ is given by
\begin{align*}
\alpha_S (G) = \inf_A \sup_{\boldsymbol{v} \in \mathcal{S}_n} \left\{ |\boldsymbol{v}|^2 \ \middle| \ \langle \boldsymbol{v}, A \boldsymbol{v} \rangle =0 \vphantom{1^{1^1}} \right\} \ \: ,
\end{align*}
where $A$ ranges over all (real symmetric) weighted adjacency matrices for $G$, and $\mathcal{S}_n$ is given in \eqref{eq:def of real sphere}.
\end{definition}
The reason behind calling $\alpha_S(G)$ the \emph{spherical independence number} is that replacing the spherical set of points $\mathcal{S}_n$ in Definition \ref{def:spherical ind number} above with the set $\{0,1\}^n \subset \mathbb{R}^n$ of vectors whose components equal only $0$ or $1$ (that is, Boolean indicator functions for subsets of $V(G)$), one naturally recovers the ordinary independence number of $G$, i.e. $\inf_A \sup_{\boldsymbol{v} \in \{0,1\}^n} \left\{ |\boldsymbol{v}|^2 \ \middle| \ \langle \boldsymbol{v}, A \boldsymbol{v} \rangle =0  \right\} = \alpha(G)$. It is also not difficult to show that if one simply replaces $\mathcal{S}_n$ with the subset of $\mathcal{S}_n$ consisting of vectors with non-negative components, then one similarly recovers $\alpha(G)$. 

\begin{remark}
Since $\{0,1\}^n \subset \mathcal{S}_n$, we clearly see that $ \alpha(G) \leq \alpha_{\mathcal{S}} (G)$ holds for all $G$.
\end{remark}

\section{Main results} \label{sec:main results}

The main result of this paper, Theorem \ref{theorem:char lovasz numb}, is the following new characterization of $\vartheta(G)$, showing the Lovasz number to be identical with the spherical independence number from Definition \ref{def:spherical ind number} and expressible in terms of minimums of walk-generating functions
\bigskip
\begin{theorem} \label{theorem:char lovasz numb}
For any graph $G$, we have
\begin{align} \label{eq:thm main thm}
\vartheta(G) \ = \ \alpha_S (G) \ = \ \inf_A \min \left\{ W_A(x) \ \middle| \ x \in \left[ \lambda_{min}(A)^{-1}, \lambda_{max}(A)^{-1} \right] \vphantom{1^{1^1}} \right\} \ \: ,
\end{align}
where in the rightmost expression above, $A$ ranges over all real symmetric weighted adjacency matrices for $G$.
\end{theorem}

\bigskip
Here, $\lambda_{min} (A), \ \lambda_{max} (A)$ are, respectively, the smallest and largest eigenvalues of $A$. If $A=0$ in the rightmost expression in \eqref{eq:thm main thm}, the expression is taken to equal $n$, corresponding to $\lambda_{min} (A)^{-1}$ and $\lambda_{max}(A)^{-1}$ being set to $- \infty$ and $+ \infty$ (otherwise $\lambda_{min} (A)$ and $\lambda_{max}(A)$ will be strictly negative and positive respectively). As shown in Appendix \ref{app:submultiplic}, it can be proven directly that the following quantity, depending on $G$ through the minimization over all weighted adjacency matrices $A$ of $G$,
\begin{align*}
\inf_A \min \left\{ W_A(x) \ \middle| \ x \in \left[ \lambda_{min}(A)^{-1}, \lambda_{max}(A)^{-1} \right] \vphantom{1^{1^1}} \right\}
\end{align*}
appearing in \eqref{eq:thm main thm} is sub-multiplicative under the strong graph product $\boxtimes$. One also straightforwardly sees that $\alpha_{\mathcal{S}} (G)$ is an upper bound on the independence number $\alpha(G)$. This, by Theorem \ref{theorem:char lovasz numb}, thus serves as an independent proof of the fact mentioned in the Introduction, that $\vartheta(G)$ upper bounds the Shannon capacity $\Theta(G)$ of $G$.  
\\
\\
The new characterization of $\vartheta(G)$ from Theorem \ref{theorem:char lovasz numb} has interesting corollaries and applications. The theorem entails several generalizations of the Hoffman upper bound on $\vartheta(G)$, which are all \emph{natural} in the sense defined in the Introduction.

\begin{figure}
    \centering
    \begin{tikzpicture}[main/.style = {draw, circle}, scale=0.8] 
\node[main] (1) at (0,1.85) {}; 

\node[main] (2) at (0,0.36) {};
\node[main] (3) at (-0.7,-0.1) {}; 
\node[main] (4) at (0.7,-0.1) {};

\node[main] (5) at (0,-0.5) {}; 

\node[main] (7) at (-0.7,-0.9) {};
\node[main] (8) at (0.7,-0.9) {};
\node[main] (6) at (0,-1.3) {};

\node[main] (9) at (-2.2,-1.7) {};
\node[main] (10) at (2.2,-1.7) {};

\draw (1) -- (9);
\draw (1) -- (10);
\draw (9) -- (10);

\draw (1) -- (2);
\draw (9) -- (7);
\draw (10) -- (8);

\draw (5) -- (2);
\draw (5) -- (3);
\draw (5) -- (4);
\draw (5) -- (6);
\draw (5) -- (7);
\draw (5) -- (8);

\draw (2) -- (3);
\draw (2) -- (4);
\draw (3) -- (7);
\draw (4) -- (8);
\draw (6) -- (7);
\draw (6) -- (8);

    \end{tikzpicture}
    \caption{ \small{A drawing of the Golomb graph.}}
    \label{fig:Golomb graph drawing}
\end{figure}
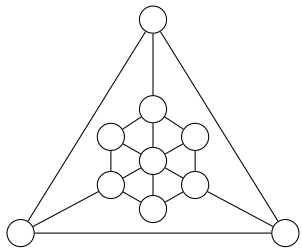

\begin{figure}
    \centering
    \includegraphics[width=0.6\linewidth]{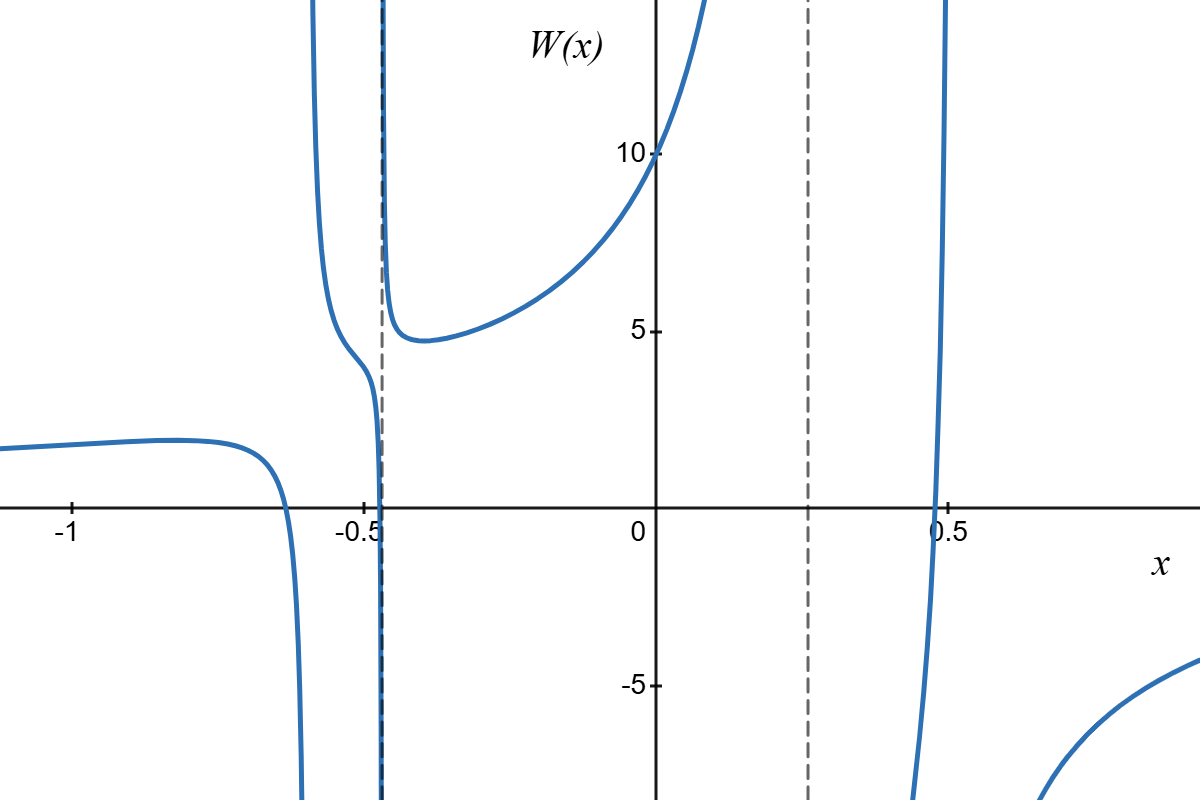}
    \caption{ \small{A plot of the (unweighted) walk-generating function $W(x)$ for the Golomb graph $G$ as a function of $x$ (See Fig. \ref{fig:Golomb graph drawing} for a drawing of the Golomb graph). The blue line is the plot of $W(x)$, while the dashed lines demarcates the interval $(\lambda_n^{-1}, \lambda_1^{-1})$. In this interval, $W(x)$ attains a minimum value of about $4.744$, which then shows $\vartheta(G) \leq 4.744$. }}
    \label{fig:Golomb graph walk-generating function}
\end{figure}

\begin{corollary}  \label{cor:precise compl generaliz}
For any graph $G$, we have
\begin{align} \label{eq:precise comple gen}
\vartheta(G) \leq \min_{\lambda_n^{-1} \leq x \leq 0} W(x) = \min_{\lambda_n^{-1} \leq x \leq 0} \sum_{j=1}^n \frac{|\langle \boldsymbol{1}, \boldsymbol{u}_j \rangle|^2}{1-\lambda_i x} \ \: ,
\end{align}
where $W(x)$ is the (un-weighted) walk-generating function of $G$, and $\{ \lambda_j \}_{j=1}^n$ are the eigenvalues of the (un-weighted) adjacency matrix for $G$ listed in decreasing order and whose corresponding eigenvectors are $\{ \boldsymbol{u}_j \}_{j=1}^n$.
\end{corollary}

\begin{proof}
The corollary follows straightforwardly from Theorem \ref{theorem:char lovasz numb}, except for the fact that we can w.l.o.g. choose our interval to end at $0$ instead of at $\lambda_1^{-1}$ in the minimization. But this follows from the convexity of $W_A(x)$ when $x$ is in the allowed interval and the fact that $W_A' (0) = \langle \boldsymbol{1}, A \boldsymbol{1} \rangle = 2 |E(G)| \geq 0$. 
\end{proof}
 
As a visualized example of the bound from Corollary \ref{cor:precise compl generaliz}, a plot of $W(x)$ is given in Fig. \ref{fig:Golomb graph walk-generating function} for the case of the \emph{Golomb graph} -- see Fig. \ref{fig:Golomb graph drawing} for a drawing. Although the formula above is easy to calculate from the entries $\{\boldsymbol{u}_j\}_j$ and $\{ \lambda_j \}_j$, since this involves minimizing a strictly convex function over a bounded interval, it is not a closed-form expression of the eigenvectors and eigenvalues. However, Corollary \ref{cor:precise compl generaliz} leads to a closed-form natural generalization of the Hoffman bound on $\vartheta(G)$ if one simply chooses $x$ in \eqref{eq:precise comple gen} to be an appropriate combination of $\lambda_n^{-1}$ and $\lambda_1^{-1}$, like e.g. $x=\frac{\langle \boldsymbol{1}, P_{\lambda_1} \boldsymbol{1} \rangle}{n} \lambda_n^{-1}+\frac{n-\langle \boldsymbol{1}, P_{\lambda_1} \boldsymbol{1} \rangle}{n} \lambda_1^{-1}$, or one can simply choose $x = \frac{\langle \boldsymbol{1}, P_{\lambda_1} \boldsymbol{1} \rangle}{n} \lambda_n^{-1} $ or $x = \frac{n-\langle \boldsymbol{1}, P_{\lambda_n} \boldsymbol{1} \rangle}{n} \lambda_n^{-1} $. One can also derive even simpler natural generalizations, as displayed in the following corollary

\begin{corollary} \label{cor:imprecise simpl generalization}
For any graph $G$, we have
\begin{align} \label{eq:corollary imprecise}
\vartheta(G) \leq \frac{-n \lambda_n}{\lambda_1-\lambda_n} \frac{\langle \boldsymbol{1}, P_{\lambda_1} \boldsymbol{1} \rangle}{n} \left(1+\sqrt{\frac{\lambda_1(n-\langle \boldsymbol{1}, P_{\lambda_1} \boldsymbol{1} \rangle)}{-\lambda_n \langle \boldsymbol{1}, P_{\lambda_1} \boldsymbol{1} \rangle} } \right)^2 \ \: ,
\end{align}
provided that we have $\frac{-\lambda_n (n-\langle \boldsymbol{1}, P_{\lambda_1} \boldsymbol{1} \rangle)}{\lambda_1 \langle \boldsymbol{1}, P_{\lambda_1} \boldsymbol{1} \rangle} \leq 1$, where $\lambda_1$ is the largest eigenvalue of the (un-weighted) adjacency matrix for $G$ with corresponding eigen-projection $P_{\lambda_1}$, and $\lambda_n$ is the smallest eigenvalue. 
\end{corollary}

\begin{proof}
By using the upper bound from Corollary \ref{cor:precise compl generaliz} and noting that for all $x$ in the interval $\lambda_n^{-1} \leq x \leq 0$, we have $\frac{1}{1-\lambda_j x} \leq \frac{1}{1-\lambda_n x}$ for any $1 \leq j \leq n$, we get
\begin{align}
\vartheta(G) \leq \min_{\lambda_n^{-1} \leq x \leq 0} \sum_{\lambda \in \sigma(A)}^n \frac{\langle \boldsymbol{1}, P_{\lambda} \boldsymbol{1} \rangle}{1-\lambda x} \leq \min_{\lambda_n^{-1} \leq x \leq 0} \left\{ \frac{\langle \boldsymbol{1}, P_{\lambda_1} \boldsymbol{1} \rangle}{1-\lambda_1 x}+ \sum_{\lambda \neq \lambda_1 } \frac{\langle \boldsymbol{1}, P_{\lambda} \boldsymbol{1} \rangle}{1-\lambda_n x} \right\} \notag \\
= \min_{\lambda_n^{-1} \leq x \leq 0} \left\{ \frac{\langle \boldsymbol{1}, P_{\lambda_1} \boldsymbol{1} \rangle}{1-\lambda_1 x}+ \frac{1}{1-\lambda_n x} \sum_{\lambda \neq \lambda_1 }^n \langle \boldsymbol{1}, P_\lambda \boldsymbol{1} \rangle  \right\} \notag \\ \label{eq:proof aproxx haffmann longline}
= \min_{\lambda_n^{-1} \leq x \leq 0} \left\{ \frac{\langle \boldsymbol{1}, P_{\lambda_1} \boldsymbol{1} \rangle}{1-\lambda_1 x}+ \frac{n-\langle \boldsymbol{1}, P_{\lambda_1} \boldsymbol{1} \rangle}{1-\lambda_n x} \right\}
\end{align}
where we have also used $\sum_{\lambda \in \sigma(A)} \langle \boldsymbol{1}, P_\lambda \boldsymbol{1} \rangle = \langle \boldsymbol{1}, \boldsymbol{1} \rangle = n$. We make the following choice for $x$ in the RHS of (\ref{eq:proof aproxx haffmann longline}) in the inequality above $$x=-\frac{1-\sqrt{\frac{-\lambda_n (n-\langle \boldsymbol{1}, P_{\lambda_1} \boldsymbol{1} \rangle)}{\lambda_1 \langle \boldsymbol{1}, P_{\lambda_1} \boldsymbol{1} \rangle}}}{-\lambda_n+\lambda_1 \sqrt{\frac{-\lambda_n (n-\langle \boldsymbol{1}, P_{\lambda_1} \boldsymbol{1} \rangle)}{\lambda_1 \langle \boldsymbol{1}, P_{\lambda_1} \boldsymbol{1} \rangle}}}$$
It can be checked that this $x$ indeed belong to the allowed interval, where one must here use the assumption $\frac{-\lambda_n (n-\langle \boldsymbol{1}, P_{\lambda_1} \boldsymbol{1} \rangle)}{\lambda_1 \langle \boldsymbol{1}, P_{\lambda_1} \boldsymbol{1} \rangle} \leq 1$, and that it in fact is the unique minimizer of the relevant function in the allowing interval. Plugging this value for $x$ into \eqref{eq:proof aproxx haffmann longline} proves the theorem.
\end{proof}

Note that by the Cauchy-Schwarz inequality, we always have $\langle \boldsymbol{1}, P_{\lambda_1} \boldsymbol{1} \rangle \leq n$, meaning that \eqref{eq:corollary imprecise} is well-defined. For a connected graph, we get $\langle \boldsymbol{1}, P_{\lambda_1} \boldsymbol{1} \rangle = |\langle \boldsymbol{1}, \boldsymbol{u}_1 \rangle|^2$ and the condition $\frac{-\lambda_n (n-\langle \boldsymbol{1}, P_{\lambda_1} \boldsymbol{1} \rangle)}{\lambda_1 \langle \boldsymbol{1}, P_{\lambda_1} \boldsymbol{1} \rangle} \leq 1$, which is equivalent to $|\langle \boldsymbol{1}, \boldsymbol{u}_1 \rangle|^2 \geq \frac{-\lambda_n n}{\lambda_1-\lambda_n}$, is then expected to be always satisfied, or at least very rarely violated. If one however wants a bound that is guaranteed to hold for all graphs, then one can by the methods used above also prove a corresponding simple bound in case $\frac{-\lambda_n (n-\langle \boldsymbol{1}, P_{\lambda_1} \boldsymbol{1} \rangle)}{\lambda_1 \langle \boldsymbol{1}, P_{\lambda_1} \boldsymbol{1} \rangle} \geq 1$ using more complicated upper bounds that deals with the case when the minimizing $x$ from \eqref{eq:proof aproxx haffmann longline} is greater than $0$.

The fact that all the mentioned natural generalizations of the Hoffman bound reduced to $\frac{-n \lambda_n}{\lambda_1 - \lambda_n}$ in the case of regular graphs follows from the fact that if $G$ is regular, we have $\langle \boldsymbol{1}, P_{\lambda_j} \boldsymbol{1} \rangle = n \delta_{1 j}$. Thus, the bound from Corollary \ref{cor:precise compl generaliz} e.g. reduces to $ \vartheta(G) \leq \min_{\lambda_n^{-1} \leq x \leq 0}  \frac{n}{1-\lambda_1 x} = \frac{n}{1-\frac{\lambda_1}{\lambda_n}}$ as desired, and similarly for the other bounds. 
\\
\\
As mentioned in the Introduction, the natural generalization of the Hoffman bound offered by Corollary \ref{cor:precise compl generaliz} will always be better or equal to the bound \eqref{eq:bachoc bound theta}. This is the content of the following proposition. 
\begin{prop} \label{prop:better Hbound}
For any graph $G$, we have 
\begin{align*}
\min_{\lambda_n^{-1} \leq x \leq 0} W(x) \leq n \left( 1 - \frac{\delta}{\mu_1}\right) \ \: ,
\end{align*}
where $W(x)$ is the (un-weighted) walk-generating function, $\delta$ is the minimum degree, $\mu_1$ is the largest Laplacian eigenvalue and $\lambda_n$ is the smallest adjacency eigenvalue of $G$.
\end{prop}

\begin{proof}
It follows from $W(x)$ being strictly convex in the interval $(\lambda_n^{-1}, \lambda_1^{-1})$ and the fact that $W ' (0) = \langle \boldsymbol{1}, A \boldsymbol{1} \rangle = 2 |E(G)| \geq 0$ that $\min_{\lambda_n^{-1} \leq x \leq 0} W(x) = \min_{\lambda_n^{-1} \leq x \leq \lambda_1^{-1}} W(x)$. It can be seen from the proof of Theorem \ref{theorem:char lovasz numb} that we have equality between the spherical independence number and the expression involving walk-generating functions even if we fix a weighted adjacency matrix $A$ instead of taking an infimum over them, i.e. the second equality in Theorem \ref{theorem:char lovasz numb} holds even if we remove $\inf_A$ from both sides. However, for our purposes here, it is sufficient to note that by proposition \ref{prop:existence of optimizer vector} there will for any $A$ always exist a real vector $\boldsymbol{u}$ satisfying $|\boldsymbol{u}|^2 = \min_{\lambda_{min}(A)^{-1} \leq x \leq \lambda_{max}(A)^{-1}} W_A (x)$, $|\boldsymbol{u}|^2 = \langle \boldsymbol{1}, \boldsymbol{u} \rangle$ and $\langle \boldsymbol{u}, A \boldsymbol{u} \rangle = 0$, which is substantially easier to prove. We here then simply take $A$ to be the unweighted adjacency matrix for $G$, and let $\boldsymbol{u}$ be such a vector satisfying the conditions just mentioned. Since the Laplacian $L= D - A$ is a real symmetric, hence Hermitian matrix, we must have
\begin{align}
\left\langle \boldsymbol{u}-\frac{|\boldsymbol{u}|^2}{n}\boldsymbol{1}, L \left( \boldsymbol{u}-\frac{|\boldsymbol{u}|^2}{n}\boldsymbol{1} \right) \right\rangle \ \leq \ &  \mu_1 \left| \boldsymbol{u}-\frac{|\boldsymbol{u}|^2}{n}\boldsymbol{1} \right|^2 = \mu_1 \left( |\boldsymbol{u}|^2 - 2 \frac{|\boldsymbol{u}|^2}{n} \langle \boldsymbol{1}, \boldsymbol{u} \rangle + \frac{|\boldsymbol{u}|^4}{n} \right) \label{eq:proof of better bound l1} \\ \label{eq:proof of better bound l2}
& \quad = \left(1-\frac{|\boldsymbol{u}|^2}{n} \right) \mu_1 |\boldsymbol{u}|^2 \ \ ,
\end{align}
where we have used $\langle \boldsymbol{1}, \boldsymbol{u} \rangle = |\boldsymbol{u}|^2$ in going from line \eqref{eq:proof of better bound l1} to \eqref{eq:proof of better bound l2}. It is easy to see that the Laplacian $L$ by construction is symmetric and annihilates the vector $\boldsymbol{1}$, i.e. $L \boldsymbol{1} = 0 $ and $\boldsymbol{1}^T L = 0$. Using this fact, we get 
\begin{align} \label{eq:proof of better bound l3}
\left\langle \boldsymbol{u}-\frac{|\boldsymbol{u}|^2}{n}\boldsymbol{1}, L \left( \boldsymbol{u}-\frac{|\boldsymbol{u}|^2}{n}\boldsymbol{1} \right) \right\rangle = \langle \boldsymbol{u}, L \boldsymbol{u} \rangle = \langle \boldsymbol{u}, (D-A) \boldsymbol{u} \rangle = \langle \boldsymbol{u}, D \boldsymbol{u} \rangle
\end{align}
where we have used that $\boldsymbol{u}$ satisfies $\langle \boldsymbol{u}, A \boldsymbol{u} \rangle = 0$. Plugging \eqref{eq:proof of better bound l3} into \eqref{eq:proof of better bound l1} gives us
\begin{align*}
\langle \boldsymbol{u}, D \boldsymbol{u} \rangle \leq \left(1-\frac{|\boldsymbol{u}|^2}{n} \right) \mu_1 |\boldsymbol{u}|^2
\end{align*}
Now, by definition of the minimum degree $\delta$, we must have $\delta I \leq D$, in semi-definite order, which when used in the inequality above gives us
\begin{align*}
\delta |\boldsymbol{u}|^2 \leq \langle \boldsymbol{u}, D \boldsymbol{u} \rangle \leq \left(1-\frac{|\boldsymbol{u}|^2}{n} \right) \mu_1 |\boldsymbol{u}|^2 \quad \Rightarrow \quad \delta \leq \left(1-\frac{|\boldsymbol{u}|^2}{n} \right) \mu_1
\end{align*}
using that $\mu_1$ is positive (since $L$ is positive definite for non empty graphs) and rearranging the inequality $\delta \leq \left(1-\frac{|\boldsymbol{u}|^2}{n} \right) \mu_1$ derived above yields $|\boldsymbol{u}|^2 \leq n \left( 1 - \frac{\delta}{\mu_1} \right)$, which when combined with the properties that $\boldsymbol{u}$ was chosen to have as discussed above, we finally get
\begin{align*}
\min_{\lambda_n^{-1} \leq x \leq 0} W(x) = \min_{\lambda_n^{-1} \leq x \leq \lambda_1^{-1}} W(x) = |\boldsymbol{u}|^2 \leq n \left( 1 - \frac{\delta}{\mu_1} \right)
\end{align*}
which proves the proposition.
\end{proof}

\begin{figure}
    \centering
    \includegraphics[width=0.7\linewidth]{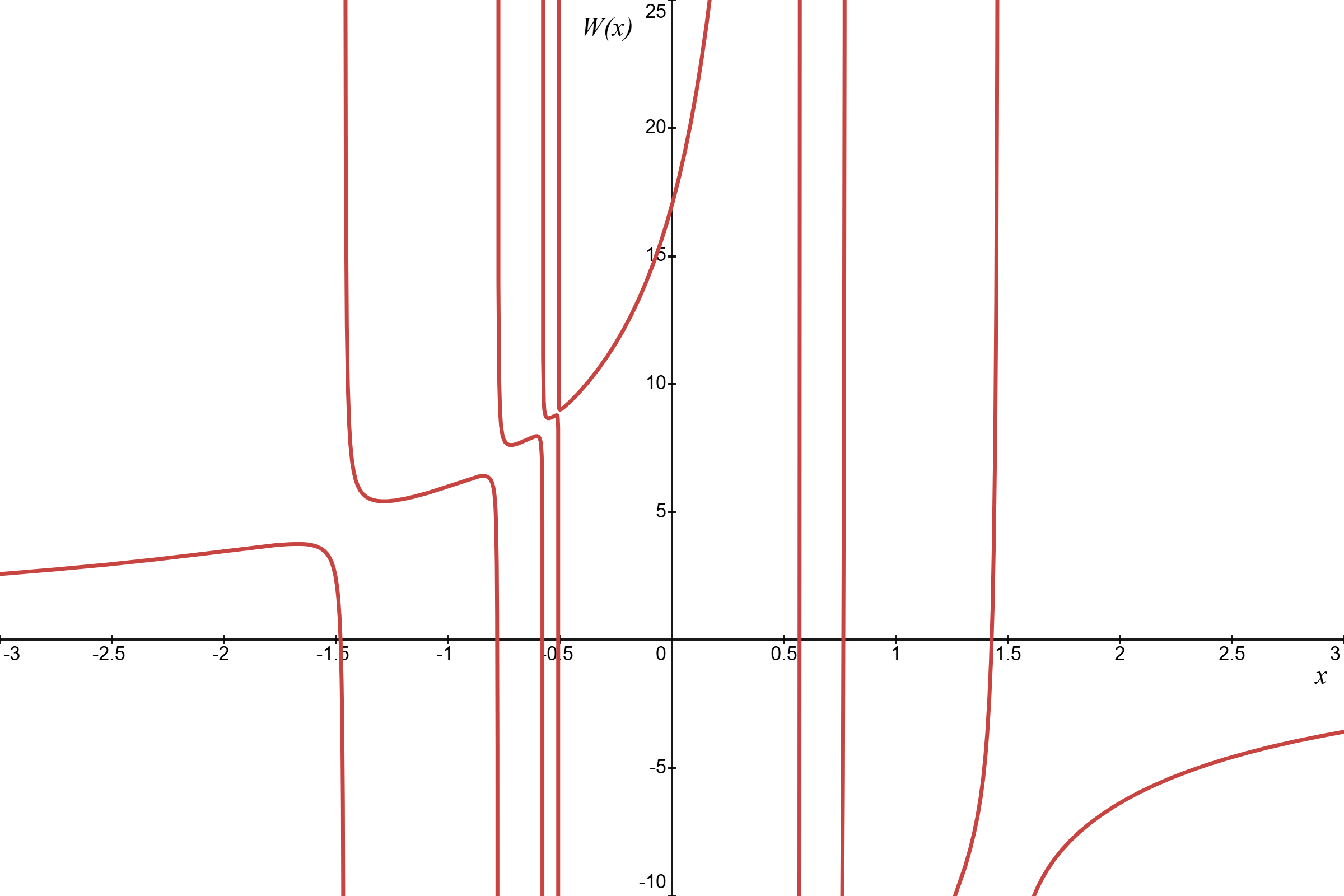}
    \caption{ \small{A plot of the (unweighted) walk-generating function $W(x)$ of the $17$ vertex path graph $P_{17}$ as a function of $x$. One sees that the critical point at which $W(x)$ attains its largest value indeed lies in the smooth interval containing $0$, i.e. $(\lambda_n^{-1}, \lambda_1^{-1})$. At this value of $x$, $W(x)=9$ which shows $\vartheta(P_{17}) \leq 9$. In fact, we here have $\alpha(P_{17}) = \vartheta(P_{17})=9$. } }
    \label{fig:P17}
\end{figure}
Lastly, we would like to call attention to an interesting analysis result, which had to be established in order to prove Theorem \ref{theorem:char lovasz numb}. Theorem \ref{lemma:app:main lemma for reciprocal functs}, proven in Appendix \ref{app:special reciprocal function}, concerns real functions $f(x)$ which can be written as finite sums of linear reciprocals functions, i.e. is of the form 
\begin{align} \label{eq:presenting reuslts reciproc funct}
f(x) = \sum_{j=1}^N \frac{\alpha_j}{1-\beta_j x}
\end{align}
with the constraint that $\alpha_j > 0$ for all $1 \leq j \leq N$. Theorem \ref{lemma:app:main lemma for reciprocal functs} then states that a function $f(x)$ of the form described above in \eqref{eq:presenting reuslts reciproc funct} has critical points, i.e. a $y \in \mathbb{R}$ for which $f'(y)=\frac{df}{dx}(y)=0$, if and only if the set $\{ \beta_j \}_{j=1}^N$ contains both strictly positive and negative elements, in which case the \emph{maximal critical point}, i.e. the critical point $y \in \mathbb{R}$ at which $f(y)$ is largest, must be located in the interval $(\beta_{min}^{-1},\beta_{max}^{-1})$, where $\beta_{min}$ and $\beta_{max}$ are the smallest and largest element of the set $\{ \beta_j \}_{j=1}^N$ respectively. In this interval, $f(x)$ is also strictly convex. This then allows us to write
\begin{align} \label{eq:presentation duality formula}
\max \left\{ f(y) \ \mid \ y \in \mathbb{R}, \ f'(y)=0  \right\} = \min \left\{ f(x) \ \mid \ \beta_{min}^{-1} \leq x \leq \beta_{max}^{-1} \right\}
\end{align}
Any weighted walk-generating function is on the form \eqref{eq:presenting reuslts reciproc funct}, and the result \eqref{eq:presentation duality formula} above can be seen to hold by inspecting Fig. \ref{fig:Golomb graph walk-generating function} in the case of the un-weighted walk-generating function for the Golomb graph. However, the result is even more apparent when e.g. looking at a plot of $W(x)$ for the $17$ vertex path graph $P_{17}$, which contains many critical points as seen in Fig. \ref{fig:P17}.

\section{Proof of the main theorem} \label{sec:proof of main results}

In this section, we prove the main theorem of this paper -- Theorem \ref{theorem:char lovasz numb}. Some tedious calculations have been moved to Appendix \ref{app:special reciprocal function}, but the arguments are besides this self-contained. Since Theorem \ref{theorem:char lovasz numb} involves two equalities, the proof of the theorem has been split up into two sections.

\subsection{Equality between Lovasz number and spherical independence number} \label{subsec:lovasz = spherical}

In this section, we prove the following lemma, showing equality between the Lovasz number and the spherical independence number.

\begin{lemma} \label{lemma:equality between lovasz and spher ind}
For all graphs $G$, we have $\vartheta(G)=\alpha_S(G)$.
\end{lemma}

 \begin{proof} 
We shall first prove $\vartheta(G) \geq \alpha_S(G)$ and then afterwards prove $\vartheta(G) \leq \alpha_S(G)$. To do this, we shall use equation \eqref{eq:lovasz number max eigen def} which states that the Lovasz theta function $\vartheta(G)$ for any $G$ equals the minimum of the largest eigenvalue of any symmetric matrix $B$, whose components $\{B_{ij}\}_{i,j=1}^n$ satisfies $B_{ij}=1$ if $i=j$ or if $i \not\sim j$ in $G$. This is equivalent to 
\begin{align} \label{eq:proof of spherical ind:formula fir vartheta}
\vartheta(G) = \inf_{A} \lambda_{max} (J-A) \ ,
\end{align}
where again $J$ is the matrix of all $1$'s and $A$ ranges over all real symmetric weighted adjacency matrices for $G$.

\vspace{1cm}
\emph{Proof of} \ \ $\vartheta(G) \ \geq \ \alpha_S(G)$:
\vspace{0.8cm}

Let $A$ be a weighted adjacency matrix for $G$. Consider now any $\boldsymbol{v} \in \mathcal{S}_n$ satisfying $\langle \boldsymbol{v}, A \boldsymbol{v} \rangle =0$. Such a $\boldsymbol{v}$ will always exist, consider e.g. the indicator function for a single vertex. Since $J-A$ is a Hermitian matrix, we have by linearity and definition of $\mathcal{S}_n$
\begin{align*}
\lambda_{max}(J-A) |\boldsymbol{v}|^2 \geq \langle \boldsymbol{v}, (J-A) \boldsymbol{v} \rangle = \langle \boldsymbol{v}, J \boldsymbol{v} \rangle =| \langle \boldsymbol{1} , \boldsymbol{v} \rangle |^2 = |\boldsymbol{v}|^4 \\
\Rightarrow \lambda_{max}(J-A) \geq |\boldsymbol{v}|^2
\end{align*}
where we have used $J=\boldsymbol{1} \boldsymbol{1}^\dagger$ and that $\boldsymbol{v} \in \mathcal{S}_n $ entails $\langle \boldsymbol{1}, \boldsymbol{v} \rangle =|\boldsymbol{v}|^2$. Since $\boldsymbol{v}$ above was arbitrary, we get 
\begin{align*}
\lambda_{max}(J-A) \geq \sup_{\boldsymbol{v} \in \mathcal{S}_n} \left\{ |\boldsymbol{v}|^2 \ \middle| \ \langle \boldsymbol{v} , A \boldsymbol{v} \rangle =0 \right\} 
\end{align*}
and taking the infimum over all real symmetric weighted adjacency matrices $A$ in the inequality above proves the result by definition \ref{def:spherical ind number} of the spherical independence number.

\vspace{1cm}
\emph{Proof of} \ \ $\vartheta(G) \ \leq \ \alpha_S(G)$:

\vspace{0.8cm}
Let $A$ again be any real symmetric weighted adjacency matrix for $G$. We shall proceed to show that there will always exist another real symmetric weighted adjacency matrix $A'$ for $G$ such that $\lambda_{max}(J-A') \leq \sup_{\boldsymbol{v} \in \mathcal{S}_n} \left\{ |\boldsymbol{v}|^2 \ \middle| \ \langle \boldsymbol{v} , A \boldsymbol{v} \rangle =0 \right\}$, which will prove the result by \eqref{eq:proof of spherical ind:formula fir vartheta} and the definitions of the infimum and spherical independence number.

Note that for any $t \in \mathbb{R}$, $t A$ is again a real symmetric weighted adjacency matrix for $G$, and consider now $\lambda_{max}(J-t A)$ as a continuous real function of $t$ -- it is continuous by the continuity of the numerical range of Hermitian matrices under Hermitian perturbations. $\lambda_{max}(J-t A)$ obtains a minimal value for some (possibly non-unique) value of $t$, since either $A=0$ and the resulting function is constant, or $A \neq 0$ in which case $A$ has strictly positive and negative eigenvalues, due to $\trace A =0$, hence $\lambda_{max}(J-t A) \rightarrow + \infty$ as $t \rightarrow \pm \infty$. Meanwhile, $\lambda_{max}(J-t A)$ is also bounded from below since $\trace(J-t A) = \trace(J)= n$.

Let then $t$ be such that $\lambda_{max}(J-t A)$ obtains a minimal value at $t$. Thus, for any $\delta t \in \mathbb{R}$, we have 
\begin{align} \label{eq:Max eig min constraint}
\lambda_{max}(J-(t+\delta t) A) \geq \lambda_{max}(J-t A).
\end{align}
We shall now construct a $\boldsymbol{v} \in \mathcal{S}_n$ satisfying $\langle \boldsymbol{v}, A \boldsymbol{v} \rangle = 0$ and $\lambda_{max}(J-t A) \leq |\boldsymbol{v}|^2$, by choosing an eigenvector $\boldsymbol{u}$ of $J-t A$ with eigenvalue $\lambda_{max}(J-t A)$ for which $\langle \boldsymbol{u}, A \boldsymbol{u} \rangle = 0$. Such a $\boldsymbol{u} \in E_{max}(J-t A)$ must indeed exists, which follows from $t$ minimizing $\lambda_{max}(J-t A)$. This can either be proven by first-order perturbation theory for Hermitian matrices or by the following argument:
\\
\\
Suppose for contradiction that we have $\langle \boldsymbol{u}, A \boldsymbol{u} \rangle \neq 0$ for all non-zero $\boldsymbol{u} \in E_{max}(J-t A)$. By considering potential linear combination of vectors in $E_{max}(J-t A)$, it is easy to show that we must then either have $\langle \boldsymbol{u}, A \boldsymbol{u} \rangle > 0$ for all $\boldsymbol{u} \in E_{max}(J-t A)$ or $\langle \boldsymbol{u}, A \boldsymbol{u} \rangle < 0$ for all $\boldsymbol{u} \in E_{max}(J-t A)$. We can w.l.o.g. assume that the first option occurs, i.e. we have $\langle \boldsymbol{u}, A \boldsymbol{u} \rangle > 0$ for all $\boldsymbol{u} \in E_{max}(J-t A)$, since all following arguments will be identical in the other case, except that we change the sign of $\delta t$. Due to the sphere $\left\{ \boldsymbol{u} \in E_{max}(J-t A) \mid | \boldsymbol{u} |^2=1  \right\}$ being compact and the map $\boldsymbol{u} \mapsto \langle \boldsymbol{u}, A \boldsymbol{u} \rangle$ being continuous, the image of this map in $\mathbb{R}$ is also compact, meaning that it is a closed and bounded interval. And since $0$ is not in this image, there must exist a constant $K > 0$ such that $\langle \boldsymbol{u}, A \boldsymbol{u} \rangle > K$ for all $\boldsymbol{u} \in E_{max}(J-t A) $ with norm $| \boldsymbol{u} |^2=1$. Hence, by linearity, for any $\boldsymbol{u} \in E_{max}(J-t A) $ we have $\langle \boldsymbol{u}, A \boldsymbol{u} \rangle \geq K |\boldsymbol{u}|^2$.

We now want to upper bound $\lambda_{max}(J-(t+\delta t) A) $ to get a contradiction from \eqref{eq:Max eig min constraint}. Note first that since $J-t A$ is Hermitian, its eigenvectors form an ortho-normal basis, hence any unit-vector $\boldsymbol{w}$ can be written as $\boldsymbol{w}=\boldsymbol{u}+\boldsymbol{u}^\perp$ with $\boldsymbol{u} \in E_{max}(J-t A)$, $\boldsymbol{u}^\perp \in E_{max}(J-t A)^\perp$ and $1=|\boldsymbol{w}|^2=|\boldsymbol{u}|^2+|\boldsymbol{u}^\perp|^2$. Thus, if we choose $\delta t$ to be positive, we have $\delta t \langle \boldsymbol{u}, A \boldsymbol{u} \rangle \geq \delta t K |\boldsymbol{u}|^2$, leading to
\begin{align}
\langle \boldsymbol{w}, (J-(t+\delta t) A) \boldsymbol{w} \rangle = \lambda_{max}|\boldsymbol{u}|^2 + \langle \boldsymbol{u}^\perp, (J-t A) \boldsymbol{u}^\perp \rangle - \delta t \langle \boldsymbol{u}+\boldsymbol{u}^\perp , A (\boldsymbol{u}+\boldsymbol{u}^\perp) \rangle \notag \\
\leq \lambda_{max}(J-t A)-(\lambda_{max}(J-tA)-\lambda_2 (J- t A)) |\boldsymbol{u}^\perp|^2 -  \delta t K |\boldsymbol{u}|^2 - \delta t \langle \boldsymbol{u}^\perp , A (2 \boldsymbol{u}+\boldsymbol{u}^\perp) \rangle \notag \\
\leq \lambda_{max}(J-t A) - C |\boldsymbol{u}^\perp|^2-\delta t K(1-|\boldsymbol{u}^\perp|^2) +\left| \delta t  || A || \, |\boldsymbol{u}^\perp| |2 \boldsymbol{u}+\boldsymbol{u}^\perp| \right| \notag \\ \label{eq:proof lovasz = spherical step 2}
\leq \lambda_{max}(J-t A) -\delta t K- (C- \delta t  K) |\boldsymbol{u}^\perp|^2+2 \delta t || A || \, |\boldsymbol{u}^\perp| 
\end{align}
where we have denoted the second largest eigenvalue of $J-t A$ as $\lambda_2(J-t A)$ and defined $C := \lambda_{max}(J-tA)-\lambda_2 (J- t A) > 0$, and used $1=|\boldsymbol{u}|^2+|\boldsymbol{u}^\perp|^2$ and the Cauchy-Schwarz inequality. For $\delta t < C / K$, the RHS of \eqref{eq:proof lovasz = spherical step 2} can be upper bounded by finding the maximum of the relevant parabola in the variable $|\boldsymbol{u}^\perp| \in [0,1]$, which then gives us that for any unit-vector $\boldsymbol{w}$ and $ C / K > \delta t > 0$, we have
\begin{align*}
\langle \boldsymbol{w}, (J-(t+\delta t) A) \boldsymbol{w} \rangle \leq \lambda_{max}(J-t A) -\delta t K+\frac{ ||A||^2 }{C-\delta t K} \delta t^2 \qquad \text{for all}  \ \boldsymbol{w}\in \mathbb{R}^n , \ |\boldsymbol{w}|=1 \\
\Rightarrow \lambda_{max}(J-(t+\delta t)A) \leq \lambda_{max}(J-t A) - K \delta t +\frac{ ||A||^2 }{C- K \delta t } \delta t^2
\end{align*}
which follows by taking a supremum over $\boldsymbol{\omega}$. However this contradicts \eqref{eq:Max eig min constraint}, since the inequality above entails that for sufficiently small positive $\delta t$, we have $\lambda_{max}(J-(t+\delta t)A) < \lambda_{max}(J-t A)$. 

Thus, we conclude that there must exists a non-zero vector $\boldsymbol{u} \in E_{max}(J-t A)$ which satisfies $\langle \boldsymbol{u}, A \boldsymbol{u} \rangle =0$. We can now re-scale $\boldsymbol{u}$ to get a vector in $\mathcal{S}_n$ by defining $\boldsymbol{v}=\frac{ \langle \boldsymbol{u}, \boldsymbol{1} \rangle}{|\boldsymbol{u}|^2} \boldsymbol{u}$. Note now that since we have $\langle \boldsymbol{u} , A \boldsymbol{u} \rangle =0$, hence also $\langle \boldsymbol{v} , t A \boldsymbol{v} \rangle= \langle \boldsymbol{v} , A' \boldsymbol{v} \rangle =0$, the fact that $(J-A') \boldsymbol{v}=\lambda_{max}(J-A') \boldsymbol{v}$ gives us 
\begin{align}
\lambda_{max}(J-A') |\boldsymbol{v}|^2 = \langle \boldsymbol{v}, (J-A') \boldsymbol{v} \rangle  = \langle \boldsymbol{v}, J \boldsymbol{v} \rangle = | \langle \boldsymbol{1}, \boldsymbol{v} \rangle |^2=|\boldsymbol{v}|^4 \notag \\ \label{eq:v=lambdaMax}
\quad \Rightarrow \ |\boldsymbol{v}|^2=\lambda_{max}(J-A') \qquad \text{for some} \ \boldsymbol{v} \in \mathcal{S}_n  \ \text{satisfying} \ \langle \boldsymbol{v} , A \boldsymbol{v} \rangle =0 
\end{align} 
Thus, for any real symmetric adjacency matrix $A$ of $G$, there exists another real symmetric adjacency matrix $A'$ s.t.
\begin{align} \label{eq:SYM inf leq sup1}
\lambda_{max}(J-A') \leq \sup_{\boldsymbol{v} \in \mathcal{S}_n} \left\{ |\boldsymbol{v}|^2 \ \middle| \ \langle \boldsymbol{v} , A \boldsymbol{v} \rangle =0 \right\} \\ \label{eq:SYM inf leq sup2}
\Rightarrow \quad  \inf_{A''} \lambda_{max}(J-A'') \leq \sup_{\boldsymbol{v} \in \mathcal{S}_n} \left\{ |\boldsymbol{v}|^2 \ \middle| \ \langle \boldsymbol{v} , A \boldsymbol{v} \rangle =0 \right\}
\end{align}
where the infimum in \eqref{eq:SYM inf leq sup2} is taken over all real symmetric adjacency matrices for $G$. The inequality (\ref{eq:SYM inf leq sup1}) follows from the definition of the infimum while the second inequality follows from the definition of the supremum and (\ref{eq:v=lambdaMax}). Taking an infimum over all real symmetric $A$ in eq. (\ref{eq:SYM inf leq sup2}) now proves $\vartheta(G) \leq \alpha_{\mathcal{S}}(G)$.
\end{proof}

\subsection{Equality between spherical independence number and the expression involving walk-generating functions} \label{subsec:spherical=walk gen}

In this section, we shall prove the following Lemma, thus proving equality between the spherical independence number and the expression from Theorem \ref{theorem:char lovasz numb} involving weighted walk-generating functions.

\begin{lemma} \label{lemma:2nd equality}
For all graphs $G$ we have
\begin{align*} 
\alpha_{\mathcal{S}}(G) = \inf_A \min \left\{ W_A(x) \ \middle| \ \lambda_{min}(A)^{-1} \leq x \leq \lambda_{max}(A)^{-1} \right\} \ \: ,
\end{align*}
where the infimum is taken over all real symmetric adjacency matrices $A$ of $G$.
\end{lemma}
The proof of Lemma \ref{lemma:2nd equality} is rather long, and we shall prove many intermediary results before finally proving Lemma \ref{lemma:2nd equality} at the end of this section. The final proof is found on page \pageref{eq:proof main lin alg opt real}. First, we make some general remarks. Let $A$ represent a real symmetric (hence Hermitian) weighted adjacency matrix for some graph $G$. Since $A$ is Hermitian, it is diagonalizable, and we can write it as $A = \sum_{i=1}^n \lambda_i \boldsymbol{u}_i \boldsymbol{u}_i^\dagger$, where the eigenvalues $\{ \lambda_i \}_{i=1}^n$ (listed with multiplicities) are guaranteed to be real numbers due to Hermiticity. We list them in decreasing order as 
\begin{align*}
\lambda_{min}=\lambda_n \leq \lambda_{n-1} \leq \cdots \leq \lambda_2 \leq \lambda_1=\lambda_{max}
\end{align*}
Since we assume all graphs $G$ to be simple, we have $A_{ii}=0$ for all $i$, giving $\trace(A)=0$. This entails that unless $A$ is the zero matrix $A=0$, we must have $\lambda_{min} < 0$ and $\lambda_{max} > 0$. It will actually not matter in the end, but since it is always possible to diagonalize a real symmetric matrix with real vectors, we can assume $\{ \boldsymbol{u}_i \}_{i=1}^n \subset \mathbb{R}^n$. Finally, we let the eigenvectors $\{ \boldsymbol{u}_i \}_{i=1}^n$ of $A$ be ortho-normal, i.e. $\langle \boldsymbol{u}_i , \boldsymbol{u}_j \rangle = \delta_{ij}$ for all $1 \leq i,j \leq n$.

We shall now prove the following result on our way to proving Lemma \ref{lemma:2nd equality}

\begin{prop} \label{lemma:linear alg optimization}
For any graph $G$ of order $n \geq 3$ and any real symmetric and non-zero adjacency matrix $A \neq 0$ of $G$, we have
\begin{align} \label{eq:linear alg optimization}
\sup_{\boldsymbol{v} \in \mathcal{S}_n} \left\{ |\boldsymbol{v}|^2 \ \middle| \ \langle \boldsymbol{v} , A \boldsymbol{v} \rangle =0  \right\} = \min \left\{ W_A(x) \ \middle| \ \lambda_{min}(A)^{-1} \leq x \leq \lambda_{max}(A)^{-1} \right\}
\end{align}
\end{prop}

The proof of Proposition \ref{lemma:linear alg optimization} is also fairly long and involves several intermediary results, namely Propositions \ref{prop:coefficients eq}, \ref{prop:2 cases for optimum}, \ref{prop:max over central strip and critical}, \ref{prop:lin alg opt lower bound} and \ref{prop:existence of optimizer vector} and Theorem \ref{lemma:app:main lemma for reciprocal functs} (the division of the proof into so many intermediary propositions has been made for the sake of clarity), which together finally proves the Proposition on page \pageref{eq:lin alg lemma proof}. But the spirit of the proof is nothing more than using Lagrange multipliers to optimize for $|\boldsymbol{v}|^2$ over all $\boldsymbol{v} \in \mathbb{R}^n$, subject to the constraints $\langle \boldsymbol{v}, A \boldsymbol{v} \rangle =0$ and $\boldsymbol{v} \in \mathcal{S}_n$ i.e. $|\boldsymbol{v}|^2=\langle \boldsymbol{1}, \boldsymbol{v} \rangle$. 

The function $|\boldsymbol{v}|^2$ is by the triangle inequality continuous in $\boldsymbol{v} \in \mathbb{R}^n$. The function $\boldsymbol{v} \mapsto \langle \boldsymbol{v}, A \boldsymbol{v} \rangle$ is similarly continuous, hence the set $\left\{ \boldsymbol{v} \in \mathbb{R}^n \ \mid \ \langle \boldsymbol{v}, A \boldsymbol{v} \rangle=0 \right\}$ is closed, since it is the preimage of $\{0\}$ under a continuous function. Since the sphere $\mathcal{S}_n$ is compact, the set $\mathcal{S}_n \cap \left\{ \boldsymbol{v} \in \mathbb{R}^n \ \mid \ \langle \boldsymbol{v}, A \boldsymbol{v} \rangle=0 \right\}$ is the intersection between a compact and a closed set, and is therefore itself compact. Therefore, by compactness and continuity, there exists a (not necessarily unique) vector $\boldsymbol{v} \in \mathcal{S}_n \cap \left\{ \boldsymbol{v} \in \mathbb{R}^n \ \mid \ \langle \boldsymbol{v}, A \boldsymbol{v} \rangle=0 \right\}$ which maximizes $|\boldsymbol{v}|^2$. We shall call this vector $\boldsymbol{v}_{op}$. We thus have
\begin{align*} 
\sup_{\boldsymbol{v} \in \mathcal{S}_n} \left\{ |\boldsymbol{v}|^2 \ \middle| \ \langle \boldsymbol{v} , A \boldsymbol{v} \rangle =0  \right\} = |\boldsymbol{v}_{op}|^2 \qquad \text{for some} \ \boldsymbol{v}_{op} \in \left\{ \boldsymbol{v} \in \mathcal{S}_n \ \mid \ \langle \boldsymbol{v}, A \boldsymbol{v} \rangle=0 \right\}
\end{align*}
We now set up the optimization problem. To make calculations easier, we shall work in the real eigen-basis of $A$ so that an arbitrary vector $\boldsymbol{v} \in \mathbb{R}^n$ can now be written as $\boldsymbol{v}=\sum_{i=1}^n \alpha_i \boldsymbol{u}_i$ for some real numbers $\alpha_1, \cdots , \alpha_n$. In this basis, the optimization problem becomes
\begin{align*} 
\text{maximize}: \ \ \sum_{i=1}^n \alpha_i^2 \ \quad \text{subject to}: \ \: \sum_{i=1}^n \lambda_i \alpha_i^2=0 \: \ \text{and} \ \: \sum_{i=1}^n \alpha_i^2 = \sum_{i=1}^n \alpha_i \langle \boldsymbol{1} , \boldsymbol{u}_i \rangle
\end{align*}
and we can write an optimal solution $\boldsymbol{v}_{op}$ (which must exist as argued above) as $\boldsymbol{v}_{op} = \sum_{i=1}^n \alpha_{op,i} \boldsymbol{u}_i$ with real eigen-coefficients $\alpha_{op,i}$. We thus have a maximization problem with $n$ real variables $\{ \alpha_i \}_{i=1}^n$ and $2$ real constraint functions. The following result provides the first step towards solving these optimization problems, i.e. towards proving Proposition \ref{lemma:linear alg optimization}. 

\begin{prop} \label{prop:coefficients eq}
Given the optimization problem presented and set up above, the optimal solution $\boldsymbol{v}_{op}$ must either satisfy $\boldsymbol{v}_{op} \in \kernel(A)$ or satisfy the following equation for some $x \in \mathbb{R}$
\begin{align*}
(1-x \lambda_i) \alpha_{op,i}= \langle \boldsymbol{1}, \boldsymbol{u}_i \rangle
\end{align*} 
\end{prop}

\begin{proof}[Proof of Proposition \ref{prop:coefficients eq}]
We must maximize the function $f(\{ \alpha_i \}_{i=1}^n)=\sum_{i=1}^n \alpha_i^2$ subject to the two constraints $h_1(\{ \alpha_i \}_{i=1}^n)=\sum_{i=1}^n \lambda_i \alpha_i^2$ and $h_2(\{ \alpha_i \}_{i=1}^n)=\sum_{i=1}^n (\alpha_i^2-\alpha_i \langle \boldsymbol{1}, \boldsymbol{u}_i \rangle )$. Since both $h_1$ and $h_2$ are smooth functions from $\mathbb{R}^n$ to $\mathbb{R}$, and since $n \geq 3$, it follows from the Regular value theorem/Preimage theorem \cite[section 9.3]{tu2010Introduction} that the set $\{ p \in \mathbb{R}^n \ \mid \ h_1(p)=0, \ h_2(p)=0 \}$ is at least locally a smooth $n-2$ dimensional manifold near a point $p$ where the matrix $(dh_p)_{ij}= \left. \frac{\partial h_i}{\partial \alpha_j} \right|_p$ is surjective, which is equivalent to the gradients $\nabla_{\boldsymbol{\alpha}} h_1$ and $\nabla_{\boldsymbol{\alpha}} h_2$ being linearly independent. The components of these gradient vectors are straightforward to calculate 
\begin{align*}
(\nabla_{\boldsymbol{\alpha}} h_1)_i = \frac{\partial h_1}{\partial \alpha_i} = 2 \lambda_i \alpha_i \quad \text{and} \quad (\nabla_{\boldsymbol{\alpha}} h_2)_i = \frac{\partial h_2}{\partial \alpha_i} = 2 \alpha_i - \langle \boldsymbol{1} , \boldsymbol{u}_i \rangle
\end{align*}
We shall now split the proof up into cases depending on whether these gradients are linearly independent or not at the point $\boldsymbol{v}_{op}$.

We first consider the case where the gradients are linearly dependent, i.e. there exists $a,b \in \mathbb{R}$, which are not both non-zero, s.t.
\begin{align} \label{eq:lin dep case reas vect}
0=a (\nabla_{\boldsymbol{\alpha}} h_1)_i(\boldsymbol{v}_{op}) + b (\nabla_{\boldsymbol{\alpha}} h_2)_i (\boldsymbol{v}_{op})= 2 a \lambda_i \alpha_{op,i} + 2 b \alpha_{op,i} - b \langle \boldsymbol{1} , \boldsymbol{u}_i \rangle \quad \text{for all $i$}
\end{align}
Using the fact that we must have $h_1 (\boldsymbol{v}_{op})=h_2 (\boldsymbol{v}_{op})=0$, i.e. $\langle \boldsymbol{v}_{op}, A \boldsymbol{v}_{op} \rangle =0$ and $|\boldsymbol{v}_{op}|^2= \langle \boldsymbol{1}, \boldsymbol{v}_{op} \rangle$, we multiply both sides of (\ref{eq:lin dep case reas vect}) by $\alpha_{op,i}$ and sum over all $1 \leq i \leq n$ to get
\begin{align*}
0 = 2a \sum_{i=1}^n \lambda_i \alpha_{op,i}^2 + 2b \sum_{i=1}^n \alpha_{op,i}^2-b \sum_{i=1}^n \langle \boldsymbol{1} , \boldsymbol{u}_i \rangle \alpha_{op,i} \\
= 2a \langle \boldsymbol{v}_{op}, A \boldsymbol{v}_{op} \rangle + 2b |\boldsymbol{v}_{op}|^2- b \langle \boldsymbol{1}, \boldsymbol{v}_{op} \rangle = b |\boldsymbol{v}_{op}|^2
\end{align*}
Thus, we must have $b=0$, since we cannot have $\boldsymbol{v}_{op}=0$, due to the assumption of $\boldsymbol{v}_{op}$ being maximizing, and a vector with a $1$ at any given vertex of $G$ and zeros otherwise would satisfy the same constraints as $\boldsymbol{v}_{op}$ but have greater norm. Since $b=0$, we must have $a \neq 0$ by assumption of not both being $0$. Plugging this into eq. (\ref{eq:lin dep case reas vect}) and dividing by $a$ gives us $\lambda_i \alpha_{op,i}=0$ for all $i$, entailing that we must have $\boldsymbol{v}_{op} \in E_{0}(A) = \kernel(A)$ in this case. 
\\
\\
We now turn to the case where the gradients $(\nabla_{\boldsymbol{\alpha}} h_1)(\boldsymbol{v}_{op})$ and $(\nabla_{\boldsymbol{\alpha}} h_2)(\boldsymbol{v}_{op})$ are linearly independent, which, as mentioned above, by the Regular value theorem/pre-image theorem entails that the set $M:=\{ p \in \mathbb{R}^n \ \mid \ h_1(p)=0, \ h_2(p)=0 \}$ is a smooth $n-2$ dimensional manifold in some neighborhood of $\boldsymbol{v}_{op}$-- we now see why we needed to assume $n \geq 3$. We can now view $f(\boldsymbol{v})=|\boldsymbol{v}|^2$ as a smooth function $f:M \rightarrow \mathbb{R}$. Since $\boldsymbol{v}_{op}$ by assumption maximizes $f$ over $M$, we must have $\nabla f|_{\boldsymbol{v}_{op}}=0$ (calculated in $M$), since otherwise one could find a vector $\boldsymbol{u} \in M$ in the smooth neighborhood near $\boldsymbol{v}_{op}$ with greater magnitude. If we calculate the derivative in $\mathbb{R}^n$, i.e. w.r.t. $\{ \alpha_i \}_{i=1}^n$, we must thus have $\nabla_{\boldsymbol{\alpha}} f |_{\boldsymbol{v}_{op}} \in \spn(\{ (\nabla_{\boldsymbol{\alpha}} h_1)(\boldsymbol{v}_{op}), (\nabla_{\boldsymbol{\alpha}} h_2)(\boldsymbol{v}_{op}) \})$, or, to phrase it in the standard Lagrange multiplier way, see e.g. \cite{fuente_mathematical_2000}, there must exist real numbers $\eta_1, \eta_2 \in \mathbb{R}$ s.t. $\nabla_{\boldsymbol{\alpha}} \left(f+ \eta_1 h_1 + \eta_2 h_2 \right)|_{\boldsymbol{v}_{op}} =0$. In components, this equation reads
\begin{align} \label{eq:lagrane mult eq sol}
2 \alpha_{op, i}+2 \eta_1 \lambda_i \alpha_{op, i} + \eta_2 (2 \alpha_{op, i}-\langle \boldsymbol{1}, \boldsymbol{u}_i \rangle) =0 \qquad \text{for all $i$}
\end{align}
Multiplying the above equation by $\alpha_{op,i}$ and summing over all $1 \leq i \leq n$, while using $|\boldsymbol{v}_{op}|^2 > 0$ as argued above, we get
\begin{align*}
0 = 2 \sum_{i=1}^n \alpha_{op,i}^2 +2 \eta_1 \sum_{i=1}^n \lambda_i \alpha_{op,i}^2+\eta_2 \left(2 \sum_{i=1}^n \alpha_{op,i}^2- \sum_{i=1}^n \langle \boldsymbol{1} , \boldsymbol{u}_i \rangle \alpha_{op,i} \right) \\
= 2 |\boldsymbol{v}_{op}|^2+2 \eta_1 \langle \boldsymbol{v}_{op}, A \boldsymbol{v}_{op} \rangle + \eta_2 (2 |\boldsymbol{v}_{op}|^2-\langle \boldsymbol{1}, \boldsymbol{v}_{op} \rangle)=(2+\eta_2) |\boldsymbol{v}_{op}|^2 \\
\Rightarrow \quad \eta_2=-2
\end{align*}
where we have again used that $\boldsymbol{v}_{op}$ must be non-zero and satisfy $\langle \boldsymbol{v}_{op}, A \boldsymbol{v}_{op} \rangle=0$ and $|\boldsymbol{v}_{op}|^2=\langle \boldsymbol{1}, \boldsymbol{v}_{op} \rangle$. Plugging $\eta_2=-2$ back into eq. (\ref{eq:lagrane mult eq sol}) and dividing by $2$ gives
\begin{align*}
(1-\eta_1 \lambda_i) \alpha_{op,i} = \langle \boldsymbol{1}, \boldsymbol{u}_i \rangle
\end{align*}
which proves the proposition with $x=\eta_1$. 
\end{proof}

From Proposition \ref{prop:coefficients eq} follows the next crucial result
\begin{prop} \label{prop:2 cases for optimum}
Suppose that the optimizer $\boldsymbol{v}_{op}$ satisfies the equation $(1-x \lambda_i) \alpha_{op,i}=\langle \boldsymbol{1}, \boldsymbol{u}_i \rangle$ for all $i$ as prescribed by the latter case from Proposition \ref{prop:coefficients eq}. Then, one of the following must be the case:
\begin{enumerate}
    \item $|\boldsymbol{v}_{op}|^2=W_A (x)$, where $x$ is a number that solves $W_A'(x)=0$
    \item $|\boldsymbol{v}_{op}|^2=W_A(\lambda^{-1})$, where $\lambda \in \sigma(A)$ is a non-zero eigenvalue of $A$ satisfying $\langle \boldsymbol{1}, P_\lambda \boldsymbol{1} \rangle =0$ and $- \lambda \ W_A '(\lambda^{-1}) \geq 0$.
\end{enumerate}
\end{prop}
We here define $W_A(\lambda^{-1})$ to equal $\sum_{\mu \in \sigma(A) \backslash \lambda} \frac{\langle \boldsymbol{1}, P_\mu \boldsymbol{1} \rangle }{1-\mu / \lambda}$ whenever $\lambda \in \sigma(A)$ and $\langle \boldsymbol{1}, P_\lambda \boldsymbol{1} \rangle =0$, i.e. we continuously extend $W_A(x)$ to inputs like $1/\lambda$ whenever it is possible as mentioned in Definition \ref{def:weighted walk-generating function}. Also, $W_A'(x)$ always denotes the $x$-derivative of $W_A$, i.e. $W_A'(x)=\frac{d W_A}{dx}(x)$.
\begin{proof}
We assume that $\boldsymbol{v}_{op}$ satisfies $(1-x \lambda_i) \alpha_{op,i}=\langle \boldsymbol{1}, \boldsymbol{u}_i \rangle$ for all $i$. There are now two cases to consider: either, $x = \lambda_j^{-1}$ for some $\lambda_j \in \sigma(A)$, or this is not the case. 

We first consider the case where $x \neq \lambda_i^{-1}$ for all $i$, meaning that we have $(1-\lambda_i x) \neq 0$ for all $i$. Thus, we can for all $i$ divide by this factor and get 
\begin{align} \label{eq:form v_op nen degenerate}
\boldsymbol{v}_{op} & = \sum_{i=1}^n \alpha_{op ,i} \boldsymbol{u}_i = \sum_{i=1}^n \frac{\langle \boldsymbol{1}, \boldsymbol{u}_i \rangle}{1-\lambda_i x} \boldsymbol{u}_i \\ \notag
\Rightarrow \quad |\boldsymbol{v}_{op}|^2 & = \langle \boldsymbol{1}, \boldsymbol{v}_{op} \rangle = \sum_{i=1}^n \frac{|\langle \boldsymbol{1}, \boldsymbol{u}_i \rangle|^2}{1-\lambda_i x} = W_A(x)
\end{align}
We remember that the eigenvectors $\boldsymbol{u}_i$ has been chosen s.t. $\langle \boldsymbol{1}, \boldsymbol{u}_{i} \rangle$ are real, hence $\langle \boldsymbol{1}, \boldsymbol{u}_{i} \rangle^2=|\langle \boldsymbol{1}, \boldsymbol{u}_{i} \rangle|^2$. $\boldsymbol{v}_{op}$ must satisfy $ \langle \boldsymbol{v}_{op}, A \boldsymbol{v}_{op} \rangle =0$, which then entails 
\begin{align} \label{eq:derivative W_A zero}
0 = \sum_{i=1}^n \lambda_i \frac{|\langle \boldsymbol{1}, \boldsymbol{u}_i \rangle|^2}{(1-\lambda_i x)^2} = W_A'(x)
\end{align}
Of course, $\boldsymbol{v}_{op}$ must also satisfy $|\boldsymbol{v}_{op}|^2=\langle \boldsymbol{1}, \boldsymbol{v}_{op} \rangle$, but this constraint is in fact automatically satisfied if $\boldsymbol{v}_{op}$ is of the form given in \eqref{eq:form v_op nen degenerate} and $x$ satisfies $W_A' (x)=0$ as prescribed by \eqref{eq:derivative W_A zero}, since then $|\boldsymbol{v}_{op}|^2-\langle \boldsymbol{1}, \boldsymbol{v}_{op} \rangle = x \langle \boldsymbol{v}_{op}, A \boldsymbol{v}_{op} \rangle = x W_A'(x)=0$. 

Hence, in case $x \neq \lambda^{-1}$ for all $\lambda \in \sigma(A)$, we indeed have $|\boldsymbol{v}_{op}|^2=W_A(x)$ for some $x$ satisfying $W_A'(x)=0$, confirming the proposition in this case. 
\\
\\
We now consider the case where $x=\lambda_x^{-1}$ for some $\lambda_x \in \sigma(A)$. Clearly, $\lambda_x$ must be non-zero. Since $\boldsymbol{v}_{op}$ by assumption satisfies $(1-x \lambda_i) \alpha_{op,i}=(1-\lambda_i / \lambda_x) \alpha_{op,i}=\langle \boldsymbol{1}, \boldsymbol{u}_i \rangle$ for all $i$, we now have an expression for all $\alpha_{op,i}$ for which $\boldsymbol{u}_i$ is not in the eigen-space $E_{\lambda_x}(A)$ associated with $\lambda_x$. Also, since $\langle \boldsymbol{1}, \boldsymbol{u}_i \rangle = (1-\lambda_i / \lambda_x) \alpha_{op,i} = 0$ for any $\boldsymbol{u}_i \in E_{\lambda_x}(A)$, i.e. for all $i$ for which $\lambda_i = \lambda_x$, we must have $\langle \boldsymbol{1}, \boldsymbol{u} \rangle =0$ for any $\boldsymbol{u} \in E_{\lambda_x}(A)$. Thus
\begin{align} \label{eq:form v_op degenerate case}
\boldsymbol{v}_{op} = \sum_{1\leq i \leq n, \ \lambda_i \neq \lambda_x} \frac{\langle \boldsymbol{1}, \boldsymbol{u}_i \rangle}{1-\lambda_i / \lambda_x} \boldsymbol{u}_i + \boldsymbol{u} = \sum_{\mu \in \sigma(A) \backslash \lambda_x} \frac{  P_\mu \boldsymbol{1}  }{1-\mu / \lambda_x}+\boldsymbol{u}
\end{align}
where $\boldsymbol{u}$ is a vector in the eigen-space associated with $\lambda_x$, i.e. $\boldsymbol{u} \in E_{\lambda_x}(A)$, and hence we have $\langle \boldsymbol{1}, \boldsymbol{u} \rangle =0$ as explained above. Thus, 
\begin{align*}
|\boldsymbol{v}_{op}|^2 = \langle \boldsymbol{1}, \boldsymbol{v}_{op} \rangle = \sum_{\mu \in \sigma(A) \backslash \lambda_x} \frac{\langle \boldsymbol{1}, P_\mu \boldsymbol{1} \rangle}{1-\mu \backslash \lambda_x} + \langle \boldsymbol{1}, \boldsymbol{u} \rangle = W_A(\lambda_x^{-1})
\end{align*}
where $W_A(\lambda_x^{-1})$ is well-defined since we have $\langle \boldsymbol{1}, P_{\lambda_x} \boldsymbol{1} \rangle =0$. We get additional conditions on $\boldsymbol{v}_{op}$ due to the constraints $|\boldsymbol{v}_{op}|^2=\langle \boldsymbol{1}, \boldsymbol{v}_{op} \rangle$ and $\langle \boldsymbol{v}_{op}, A \boldsymbol{v}_{op} \rangle=0$, which, when using the form of $\boldsymbol{v}_{op}$ from \eqref{eq:form v_op degenerate case}, the first constraint becomes
\begin{align*}
|\boldsymbol{u}|^2+\sum_{\mu \in \sigma(A) \backslash \lambda_x} \frac{\langle \boldsymbol{1}, P_\mu \boldsymbol{1} \rangle}{(1-\mu / \lambda_x)^2} = \sum_{\mu \in \sigma(A) \backslash \lambda_x} \frac{\langle \boldsymbol{1}, P_\mu \boldsymbol{1} \rangle}{1-\mu / \lambda_x} \\
\Rightarrow |\boldsymbol{u}|^2 = \sum_{\mu \in \sigma(A) \backslash \lambda_x} \left( 1-\frac{\mu}{\lambda_x}-1 \right) \frac{\langle \boldsymbol{1}, P_\mu \boldsymbol{1} \rangle}{(1-\mu / \lambda_x)^2} = -\frac{1}{\lambda_x} W_A'(\lambda_x^{-1}) \qquad \text{and} \\
\sum_{\mu \in \sigma(A) \backslash \lambda_x} \mu \frac{\langle \boldsymbol{1}, P_\mu \boldsymbol{1} \rangle}{(1-\mu / \lambda_x)^2}+\lambda_x |\boldsymbol{u}|^2 = 0 = W_A' (\lambda_x^{-1}) + \lambda_x |\boldsymbol{u}|^2 
\end{align*}
We thus see that all constraints can be satisfied by simply picking any $\boldsymbol{u} \in E_{\lambda_x}(A)$ with magnitude $-\frac{1}{\lambda_x} W_A'(\lambda_x^{-1})$, which is possible if and only if $-\frac{1}{\lambda_x} W_A'(\lambda_x^{-1}) \geq 0$, or equivalently, $-\lambda_x W_A'(\lambda_x^{-1}) \geq 0$, since $\lambda_x \neq 0$. 

Thus, in case $x=\lambda^{-1}$ for some $\lambda \in \sigma(A)$, we indeed have that $\lambda$ satisfies $\lambda \neq 0$, $\langle \boldsymbol{1}, P_\lambda \boldsymbol{1} \rangle =0$ and $ \lambda W_A'(\lambda^{-1}) \geq 0$, while $|\boldsymbol{v}_{op}|^2$ is given by $|\boldsymbol{v}_{op}|^2=W_A(\lambda^{-1})$. Thus, Proposition \ref{prop:2 cases for optimum} holds in all cases. 
\end{proof}

Note that since $W_A(x)$ is by definition invariant under which basis we choose, we no longer have to assume that the eigenvectors $\{ \boldsymbol{u}_i \}$ of $A$ are chosen to be real. We continue to the next proposition

\begin{prop}\label{prop:max over central strip and critical}
Let $\boldsymbol{v}_{op}$ be the optimizer from Proposition \ref{prop:2 cases for optimum}. We must then have
\begin{align*}
|\boldsymbol{v}_{op}|^2 \ \leq \ 
\max \left\{ \begin{matrix}
\min \left\{ \ W_A(x) \ \mid \ \lambda_{min}(A)^{-1} \leq x \leq \lambda_{max}(A)^{-1} \right\}, \\
\max\{ \ W_A(c) \quad | \ \ c \in \mathbb{R}, \ \ W_A'(c)=0 \} \\
\end{matrix} \right\}
\end{align*}
where we define $\max\{  W_A(c) \ |  \ c \in \mathbb{R}, \ W_A'(c)=0 \} := \emptyset$ if $W_A$ has no critical points.
\end{prop}

\begin{proof}
It follows straightforwardly from Proposition \ref{prop:2 cases for optimum} that $|\boldsymbol{v}_{op}|^2$ is upper bounded by the maximum of $W_A(x)$, where $x$ ranges over all critical points of $W_A$, i.e. $x \in \mathbb{R}$ for which $W_A' (x)=0$, and over the reciprocals of some special eigenvalues of $A$. We specifically have
\begin{align*}
|\boldsymbol{v}_{op}|^2 \ \leq \ 
\max \left\{ \begin{matrix}
W_A(\lambda^{-1}) \quad | \ \ \lambda \in \sigma(A), \ \lambda \neq 0, \ \langle \boldsymbol{1}, P_\lambda \boldsymbol{1} \rangle =0, \ -\lambda W_A' (\lambda^{-1}) \geq 0 \\
W_A(c) \quad | \ \ c \in \mathbb{R}, \ \ W_A'(c)=0 \\
\end{matrix} \right\}
\end{align*}
Let us now define what we shall call the \emph{central strip} of the function $W_A(x)$, as the interval $(\lambda_-^{-1},\lambda_+^{-1}) \subset \mathbb{R}$, where $\lambda_-$ ($\lambda_+$) is the smallest (largest) eigenvalue $\lambda \in \sigma(A)$ of $A$ for which $\langle \boldsymbol{1}, P_{\lambda} \boldsymbol{1} \rangle \neq 0$. If $A$ has no strictly negative (positive) eigenvalues $\lambda$ for which $\langle \boldsymbol{1}, P_{\lambda} \boldsymbol{1} \rangle \neq 0$, we set $\lambda_-:=-\infty$ ($\lambda_+:=+\infty$). For example, if both $\langle \boldsymbol{1}, P_{\lambda_{min}} \boldsymbol{1} \rangle \neq 0$ and $\langle \boldsymbol{1}, P_{\lambda_{max}} \boldsymbol{1} \rangle \neq 0$, the central strip of $W_A(x)$ is $(\lambda_{min}(A)^{-1},\lambda_{max}(A)^{-1})$
\\
\\
Note that the minimum $\min \left\{ \ W_A(x) \ \mid \ \lambda_{min}(A)^{-1} \leq x \leq \lambda_{max}(A)^{-1} \right\}$ does indeed exists, since $W_A(x)$ is convex, lower bounded and continuous in its central strip (except possibly at $x=\lambda_{min}(A)^{-1}$ or $x=\lambda_{max}(A)^{-1}$ in which case $W_A(x)$ approaches $+\infty$ at these points from within the interval), and the central strip contains the interval $\left[\lambda_{min}(A)^{-1}, \lambda_{max}(A)^{-1} \right]$. We shall now prove Proposition \ref{prop:max over central strip and critical} by showing that if $\lambda \in \sigma(A)$ is an eigenvalue of $A$ satisfying $\lambda \neq 0$, $\langle \boldsymbol{1}, P_\lambda \boldsymbol{1} \rangle =0$ and $-\lambda W_A'(\lambda^{-1}) \geq 0$, then either $W_A(\lambda^{-1}) \leq \min \left\{ \ W_A(x) \ \mid \ \lambda_{min}(A)^{-1} \leq x \leq \lambda_{max}(A)^{-1} \right\}$ or $W_A(\lambda^{-1}) \leq \max\{ \ W_A(c) \quad | \ \ c \in \mathbb{R}, \ \ W_A'(c)=0 \}$. We shall split the proof up into three cases, depending on whether $\lambda^{-1}$ is contained in the central strip $(\lambda_-^{-1},\lambda_+^{-1})$ of $W_A(x)$, or lies below or above it, that is, whether $\lambda^{-1} \leq \lambda_-^{-1}$, $\lambda^{-1} \in [\lambda_-^{-1}, \lambda_+^{-1}]$ or $\lambda^{-1} \geq \lambda_+^{-1}$. 
\\
\\
\emph{case:} \ $\lambda^{-1} < \lambda_-^{-1}$

For $\lambda^{-1} < \lambda_-^{-1}$ and $\lambda \neq 0$ to be the case, we cannot have $\lambda_-^{-1}=-\infty$, thus we must have $\lambda_- \in \sigma(A)$, $\lambda_- < 0$ and $\langle \boldsymbol{1}, P_{\lambda_-} \boldsymbol{1} \rangle \neq 0$. We also have $\lambda_- < \lambda < 0$ since $\lambda \neq 0$. We now let $\mu_-$ denote the largest eigenvalue $\mu_- \in \sigma(A)$ smaller than $\lambda$, for which $\langle \boldsymbol{1}, P_{\mu_-} \boldsymbol{1} \rangle \neq 0$. We know that such a $\mu_-$ exist, since either $\lambda_-$ satisfies all these conditions, or some larger eigenvalue does. Since $A$ by assumption has no eigenvalues $\eta \in \sigma(A)$ in the interval $\mu_- < \eta < \lambda$ satisfying $\langle \boldsymbol{1}, P_{\eta} \boldsymbol{1} \rangle \neq 0$, $W_A(x)$ is smooth in the interval $x \in (\lambda^{-1},\mu_-^{-1})$, since it is a finite sum of functions that are all smooth in this interval. Meanwhile, $W_A(x)$ contains by assumption the term $\frac{\langle \boldsymbol{1}, P_{\mu_-} \boldsymbol{1} \rangle}{1-x \mu_-}$ where  $\langle \boldsymbol{1}, P_{\mu_-} \boldsymbol{1} \rangle > 0$ and $\mu_- < 0$, which means that we have both $W_A(x) \rightarrow - \infty$ and $W_A' (x) \rightarrow - \infty$ as $x$ approaches $\mu_{-}^{-1}$ from below. Now, by assumption on $\lambda$, we have $-\lambda W_A' (\lambda^{-1}) > 0 \ \Rightarrow W_A' (\lambda^{-1}) \geq 0 $, since $\lambda < 0$. This means that either $\lambda^{-1}$ is a critical point, i.e. $W_A ' (\lambda^{-1}) =0$, or it follows by $W_A(x)$ being smooth and $W_A' (x) \rightarrow - \infty$ as $x$ approaches $\mu_{-}^{-1}$ from below that there exists a critical point $c$, $W_A' (c)=0$ in the interval $c \in (\lambda^{-1},\mu_-^{-1})$ such that $W_A(c) \geq W_A(\lambda^{-1})$, since $c$ can be chosen (as the first critical point larger than $\lambda^{-1}$) such that $W_A' (x) \geq 0$ for all $\lambda^{-1} \leq x \leq c$. Hence, we have $W_A(\lambda^{-1}) \leq \max\{ \ W_A(c) \quad | \ \ c \in \mathbb{R}, \ \ W_A'(c)=0 \}$, confirming the proposition in this case. 
\\
\\
\emph{case:} \ $\lambda^{-1} > \lambda_+^{-1}$

This case is very similar to the previous case discussed above. Similarly to before, we must now have $0 < \lambda < \lambda_+$ and $\lambda_+ \in \sigma(A)$. We let $\mu_+$ denote the smallest eigenvalue $\mu_+ \in \sigma(A)$ larger than $\lambda$, for which $\langle \boldsymbol{1}, P_{\mu_+} \boldsymbol{1} \rangle \neq 0$. Such a $\mu_+$ exist, since $\lambda_+$ satisfies all these conditions, or some smaller eigenvalue does. Similarly to before, we get that $W_A(x)$ must be smooth in the interval $x \in (\mu_+^{-1}, \lambda^{-1})$. We here get $W_A(x) \rightarrow - \infty$ and $W_A' (x) \rightarrow + \infty$ as $x$ approaches $\mu_+^{-1}$ from above, due to $W_A(x)$ containing the term $\frac{\langle \boldsymbol{1}, P_{\mu_+} \boldsymbol{1} \rangle}{1-x \mu_+}$, with $\mu_+ > 0$, $\langle \boldsymbol{1}, P_{\mu_+} \boldsymbol{1} \rangle > 0$. By assumption on $\lambda$, we must again have $- \lambda W_A' (\lambda^{-1}) \geq 0 \Rightarrow W_A' (\lambda^{-1}) \leq 0$. Thus, by smoothness of $W_A(x)$ in the interval $(\mu_+,\lambda^{-1}]$, and using $W_A' (x) \rightarrow + \infty$ as $x$ approaches $\mu_+^{-1}$ from above, there must exists a critical point $c$, $W_A' (c)=0$ in the interval $c \in (\mu_+^{-1}, \lambda^{-1}]$ for which $W_A(c) \geq W_A(\lambda^{-1})$, which proofs the proposition in this case as well.
\\
\\
\emph{case:} \ $\lambda^{-1} \in [\lambda_-^{-1}, \lambda_+^{-1}]$

By assumption of $\lambda, \lambda_-, \lambda_+$, we cannot have $\lambda=\lambda_-$ or $\lambda=\lambda_+$, since we have $\langle \boldsymbol{1}, P_\lambda \boldsymbol{1} \rangle =0$, $\lambda \neq 0$, while $\lambda_-, \lambda_+$ satisfies $\langle \boldsymbol{1}, P_{\lambda_-} \boldsymbol{1} \rangle \neq 0$, $\langle \boldsymbol{1}, P_{\lambda_+} \boldsymbol{1} \rangle \neq 0$ or one of or both $\lambda_-$, $\lambda_+$ equals $\pm \infty$. In any case, we can assume $\lambda^{-1} \in (\lambda_-^{-1},\lambda_+^{-1})$, i.e. $\lambda_-^{-1} < \lambda^{-1} < \lambda_+^{-1}$.

We shall now consider two cases, depending on whether $\lambda$ is positive or negative. Since $\lambda$ by assumption satisfies $-\lambda W_A' (\lambda^{-1}) \geq 0$, if we assume $\lambda$ to be negative (positive), then we must have $W_A' (\lambda^{-1}) \geq 0$ ($W_A' (\lambda^{-1}) \leq 0$). By definition of $\lambda_{min}$ ($\lambda_{max}$) as the smallest (largest) eigenvalue, if $\lambda$ is a negative (positive) eigenvalue and we have $\lambda^{-1} \in (\lambda_-^{-1},\lambda_+^{-1})$, then we must also have $\lambda_{min}(A)^{-1} \in (\lambda_-^{-1},\lambda_+^{-1})$ ($\lambda_{max}(A)^{-1} \in (\lambda_-^{-1},\lambda_+^{-1})$), since $\lambda^{-1} \leq \lambda_{min}(A)^{-1}$ ($\lambda^{-1} \geq \lambda_{max}(A)^{-1}$), where we have used $\lambda_{min} < 0$ ($\lambda_{max} > 0$) and the definition of the minimal (maximal) eigenvalue. The case $\lambda_{min}=0$ ($\lambda_{max}=0$) is ruled out since by assumption we have $A \neq 0$ and $\trace A =0$. It is easy to show that in the central strip $x \in (\lambda_{-}^{-1}, \lambda_{+}^{-1})$, $W_A''(x) \geq 0$, i.e. $W_A(x)$ is a convex function, which means that since $W_A' (\lambda^{-1}) \geq 0$ ($W_A' (\lambda^{-1}) \leq 0$) and since $\lambda^{-1} \leq \lambda_{min}(A)^{-1}$ ($\lambda^{-1} \geq \lambda_{max}(A)^{-1}$), we must have $W_A(\lambda^{-1}) \leq W_A(\lambda_{min}(A)^{-1})$ ($W_A(\lambda^{-1}) \leq W_A(\lambda_{max}(A)^{-1})$). By convexity of $W_A(x)$ in the central strip, the fact that we have $W_A'(\lambda^{-1}) \geq 0$ ($W_A'(\lambda^{-1}) \leq 0$) also entails $W_A'(\lambda_{min}(A)^{-1}) \geq 0$ ($W_A'(\lambda_{max}(A)^{-1}) \leq 0$). This, by yet again using convexity of $W_A(x)$ in the central strip, entails $W_A(\lambda_{min}(A)^{-1}) = \min \{ W_A(x) \ \mid \ \lambda_{min}(A)^{-1} \leq x \leq \lambda_{max}(A)^{-1} \}$ ($W_A(\lambda_{max}(A)^{-1}) = \min \{ W_A(x) \ \mid \ \lambda_{min}(A)^{-1} \leq x \leq \lambda_{max}(A)^{-1} \}$ ), since in both cases (i.e. regardless of $\lambda$ being positive or negative) we have $[\lambda_{min}(A)^{-1}, \lambda_{max}(A)^{-1}] \subseteq [\lambda_{-}^{-1},\lambda_{+}^{-1}]$.

Hence, by using $W_A(\lambda^{-1}) \leq W_A(\lambda_{min}(A)^{-1})$ \quad ($W_A(\lambda^{-1}) \leq W_A(\lambda_{max}(A)^{-1})$) together with $$W_A(\lambda_{min}(A)^{-1}) = \min \{ W_A(x) \ \mid \ \lambda_{min}(A)^{-1} \leq x \leq \lambda_{max}(A)^{-1} \}$$ ($W_A(\lambda_{max}(A)^{-1}) = \min \{ W_A(x) \ \mid \ \lambda_{min}(A)^{-1} \leq x \leq \lambda_{max}(A)^{-1} \}$), we get that regardless of $\lambda$ is negative or positive, it holds that $W_A(\lambda^{-1}) \leq \min \{ W_A(x) \ \mid \ \lambda_{min}(A)^{-1} \leq x \leq \lambda_{max}(A)^{-1} \}$, which confirms the proposition in case $\lambda^{-1} \in [\lambda_-^{-1}, \lambda_+^{-1}]$
\end{proof}

We are now ready to prove inequality in one direction in (\ref{eq:linear alg optimization}) from Proposition \ref{lemma:linear alg optimization}.

\begin{prop} \label{prop:lin alg opt lower bound}
For any graph $G$ of order $n \geq 3$ with non-zero real symmetric weighted adjacency matrix $A$, we have 
\begin{align*} 
\sup_{\boldsymbol{v} \in \mathcal{S}_n} \left\{ |\boldsymbol{v}|^2 \ \middle| \ \langle \boldsymbol{v} , A \boldsymbol{v} \rangle =0  \right\} \leq \min \left\{ W_A(x) \ \middle| \ \lambda_{min}(A)^{-1} \leq x \leq \lambda_{max}(A)^{-1}(A) \right\}
\end{align*}
\end{prop}

\begin{proof}
Note that by combining Propositions \ref{prop:coefficients eq}, \ref{prop:2 cases for optimum} and \ref{prop:max over central strip and critical}, we either have $\boldsymbol{v}_{op} \in \kernel(A)$, or $ |\boldsymbol{v}_{op}|^2 \ \leq \ 
\max \left\{ \begin{matrix}
\min \left\{ \ W_A(x) \ \mid \ \lambda_{min}(A)^{-1} \leq x \leq \lambda_{max}(A)^{-1} \right\}, \\
\max\{ \ W_A(c) \quad | \ \ c \in \mathbb{R}, \ \ W_A'(c)=0 \} \\
\end{matrix} \right\} $. If $\boldsymbol{v}_{op} \in \kernel(A)$, then $P_0 \boldsymbol{v}_{op}=\boldsymbol{v}_{op}$, where $P_0$ is the projection onto the $0$ eigen-space/kernel of $A$ (in case $\kernel(A)=\{0\}$, we simply have $P_0=0$). This means that since $|\boldsymbol{v}_{op}|^2 =  \langle \boldsymbol{1}, \boldsymbol{v}_{op} \rangle$, we get by using the Cauchy-Schwarz inequality and the fact that $P_0$ is a Hermitian operator and a projection 
\begin{align}
|\boldsymbol{v}_{op}|^4 = | \langle \boldsymbol{1}, \boldsymbol{v}_{op} \rangle |^2 = | \langle \boldsymbol{1}, P_0 \boldsymbol{v}_{op} \rangle |^2 = |\langle \boldsymbol{1} P_0, \boldsymbol{v}_{op} \rangle |^2 \leq |\boldsymbol{v}_{op}|^2 |P_0 \boldsymbol{1}|^2 = |\boldsymbol{v}_{op}|^2 \langle \boldsymbol{1}, P_0 \boldsymbol{1} \rangle \notag \\ \label{eq:size vop kernel case}
\Rightarrow \qquad |\boldsymbol{v}_{op}|^2 \leq \langle \boldsymbol{1}, P_0 \boldsymbol{1} \rangle 
\end{align}
Note now that by definition of $W_A(x)$, we have that for any $x \in [\lambda_{min}(A)^{-1}, \lambda_{max}(A)^{-1}]$, we have 
\begin{align*}
W_A(x) = \sum_{\lambda \in \sigma(A)} \frac{\langle \boldsymbol{1}, P_\lambda \boldsymbol{1} \rangle}{1-\lambda x } = \langle \boldsymbol{1}, P_0 \boldsymbol{1} \rangle + \sum_{\lambda \in \sigma(A), \ \lambda \neq 0} \frac{\langle \boldsymbol{1}, P_\lambda \boldsymbol{1} \rangle}{1-\lambda x } \geq \langle \boldsymbol{1}, P_0 \boldsymbol{1} \rangle
\end{align*}
since $\frac{\langle \boldsymbol{1}, P_\lambda \boldsymbol{1} \rangle}{1-\lambda x }$ is non-negative for any eigenvalue $\lambda$ if $\lambda_{min}(A)^{-1} \leq x \leq \lambda_{max}(A)^{-1}$. We must thus always have $\langle \boldsymbol{1}, P_0 \boldsymbol{1} \rangle \leq \min \left\{ \ W_A(x) \ \mid \ \lambda_{min}(A)^{-1} \leq x \leq \lambda_{max}(A)^{-1} \right\}$. Combining this with eq. (\ref{eq:size vop kernel case}) gives us that in case $\boldsymbol{v}_{op} \in \kernel(A)$
\begin{align*}
|\boldsymbol{v}_{op}|^2 \leq \langle \boldsymbol{1}, P_0 \boldsymbol{1} \rangle \leq \min \left\{ \ W_A(x) \ \mid \ \lambda_{min}(A)^{-1} \leq x \leq \lambda_{max}(A)^{-1} \right\}
\end{align*}
Thus, in any case, it holds that 
\begin{align} \label{eq:proof prop 4 last vupper}
|\boldsymbol{v}_{op}|^2 \ \leq \ 
\max \left\{ \begin{matrix}
\min \left\{ \ W_A(x) \ \mid \ \lambda_{min}(A)^{-1} \leq x \leq \lambda_{max}(A)^{-1} \right\}, \\
\max\{ \ W_A(c) \quad | \ \ c \in \mathbb{R}, \ \ W_A'(c)=0 \} \\
\end{matrix} \right\}
\end{align}
To proceed, we shall use the following interesting analysis result, stating that if $f(x)$ is a real function of the following form 
\begin{align} \label{eq:proof main thm reciprocal form}
f(x) = \sum_{j=1}^N \frac{a_j}{1-b_j x}
\end{align}
where $N$ is a natural number, $a_j, b_j \in \mathbb{R}$ and $a_j \geq 0$ for all $1 \leq j \leq N$, then assuming that $f(x)$ has critical points, the critical point $c$, $f'(c)=0$, at which $f(c)$ takes its largest value is found in the interval $(b_-^{-1}, b_{+}^{-1})$, where $b_-$ ($b_+$) is the smallest (largest) $b_j$ for which $a_j \neq 0$. It is also easy to show that $f(x)$ must be convex in this interval, hence we can write
\begin{align*}
\max \left\{ f(c) \mid c \in \mathbb{R}, \ f'(c)=0 \right\} = \min \left\{ f(x) \mid x \in [b_-^{-1}, b_+^{-1}]  \right\}
\end{align*}
This is the content of Theorem \ref{lemma:app:main lemma for reciprocal functs} proven in Appendix \ref{app:special reciprocal function} (if this result has already been established, then it is beyond the knowledge of the author). See e.g. Figures \ref{fig:Golomb graph walk-generating function} and \ref{fig:P17} for visualizations of this result. Since $W_A(x)$ is indeed of the form (\ref{eq:proof main thm reciprocal form}), due to $\langle \boldsymbol{1} , P_\lambda \boldsymbol{1} \rangle \geq 0$ for any projection $P_\lambda$, we must always have that supposing $W_A$ has critical points, then
\begin{align*}
\max\{ \ W_A(c) \quad | \ \ c \in \mathbb{R}, \ \ W_A'(c)=0 \} = \min \left\{ \ W_A(x) \ \mid \ \lambda_{-}(A)^{-1} \leq x \leq \lambda_{+}(A)^{-1} \right\} \\
\leq \min \left\{ \ W_A(x) \ \mid \ \lambda_{min}(A)^{-1} \leq x \leq \lambda_{max}(A)^{-1} \right\}
\end{align*}
since $\left[ \lambda_{min}(A)^{-1}, \lambda_{max}(A)^{-1} \right] \subseteq \left[ \lambda_-^{-1}, \lambda_+^{-1} \right]$, which when plugged into \eqref{eq:proof prop 4 last vupper} gives 
\begin{align*}
|\boldsymbol{v}_{op}|^2 \leq \min \left\{ \ W_A(x) \ \mid \ \lambda_{min}(A)^{-1} \leq x \leq \lambda_{max}(A)^{-1} \right\}
\end{align*}
We have here assumed that $W_A(x)$ has critical points. If $W_A(x)$ has no critical point, then the proposition follows straightforwardly from \eqref{eq:proof prop 4 last vupper}. This proves the proposition in any case.  
\end{proof}

We now prove inequality in (\ref{eq:linear alg optimization}) in the other direction, which is much easier

\begin{prop} \label{prop:existence of optimizer vector}
Let $y \in \mathbb{R}$ be such that $W_A(y)=\min \left\{ W_A(x) \ \middle| \ \lambda_{min}(A)^{-1} \leq x \leq \lambda_{max}(A)^{-1} \right\}$ where $A$ is any non-zero real symmetric weighted adjacency matrix. Then, there exists a vector $\boldsymbol{u} \in \mathbb{R}^n$ satisfying 
\begin{align*}
|\boldsymbol{u}|^2=W_A(y) \ , \qquad \langle \boldsymbol{u} , A \boldsymbol{u} \rangle = 0 \ , \qquad |\boldsymbol{u}|^2 = \langle \boldsymbol{1} , \boldsymbol{u} \rangle
\end{align*}
\end{prop}

\begin{proof}
We shall consider different cases in the proof, either $y = \lambda_{min}(A)^{-1}$, $y = \lambda_{max}(A)^{-1}$ or $y \neq  \lambda_{min}(A)^{-1}, \lambda_{max}(A)^{-1} $.
\\
\\
\emph{Case: } $y = \lambda_{min}(A)^{-1}$ \quad ($y = \lambda_{max}(A)^{-1}$)
\\
\\
Assume $y = \lambda_{min}(A)^{-1}$ ($y = \lambda_{max}(A)^{-1}$). By assumption of $y$ being a minimum of $W_A(x)$ in the interval $x \in \left[ \lambda_{min}(A)^{-1}, \lambda_{max}(A)^{-1} \right]$, and since $W_A(x)$ is smooth and convex inside this interval, we must have $W_A' (y) \geq 0$ ($W_A' (y) \leq 0$). Further, for $W_A(y)$ not to equal $+\infty$ (in which case it would certainly not be a minimum), we must also have $\langle \boldsymbol{1}, P_{min} \boldsymbol{1} \rangle =0$ ($\langle \boldsymbol{1}, P_{max} \boldsymbol{1} \rangle =0$), where $P_{min}$ ($P_{max}$) is the projection onto the eigenspace of $A$ associated with $\lambda_{min}$ ($\lambda_{max}$). We then choose $\boldsymbol{u}$ to equal
\begin{align} \label{eq:proof main thm lin alg otp upperB choice u deg case}
\boldsymbol{u}=\sum_{\lambda \in \sigma(A), \ \lambda \neq y^{-1}} \frac{P_\lambda \boldsymbol{1} }{1- \lambda y} + \sqrt{- y W_A' (y)} \boldsymbol{w}
\end{align}
Where $\boldsymbol{w}$ is an arbitrary unit-vector in $E_{min}(A)$ ($E_{max}(A)$) -- the eigenspace associated with $\lambda_{min}$ ($\lambda_{max}$). Note that $\sqrt{- y W_A' (y)}$ is a real number, since $W_A' (y) \geq 0$ and $y=\lambda_{min}(A)^{-1} < 0$ ($W_A' (y) \leq 0$ and $y=\lambda_{max}(A)^{-1} > 0$). We have also chosen the eigenvectors of $A$ to be real vectors, in case  $A$ is real symmetric, which makes $\boldsymbol{u}$ a real vector in this case. When $\boldsymbol{u}$ has the form (\ref{eq:proof main thm lin alg otp upperB choice u deg case}), it indeed satisfies the properties from the proposition.

Since $\langle \boldsymbol{1}, P_{min} \boldsymbol{1} \rangle =0$ ($\langle \boldsymbol{1}, P_{max} \boldsymbol{1} \rangle =0$), we have $\langle \boldsymbol{1}, \boldsymbol{w} \rangle = 0$, and also using that the eigen-basis of $A$ is orthogonal gives us 
\begin{align*}
\langle \boldsymbol{1}, \boldsymbol{u} \rangle = \sum_{\lambda \in \sigma(A), \ \lambda \neq y^{-1}} \frac{\langle \boldsymbol{1}, P_\lambda \boldsymbol{1} \rangle }{1- \lambda y} = W_A (y) \\
\langle \boldsymbol{u} , A \boldsymbol{u} \rangle = \sum_{\lambda \in \sigma(A), \ \lambda \neq y^{-1}} \lambda \frac{\langle \boldsymbol{1}, P_\lambda \boldsymbol{1} \rangle }{(1- \lambda y)^2} - \frac{y}{y} W_A' (y) = 0 \\
|\boldsymbol{u}|^2 = \sum_{\lambda \in \sigma(A), \ \lambda \neq y^{-1}} \frac{\langle \boldsymbol{1}, P_\lambda \boldsymbol{1} \rangle }{(1- \lambda y)^2} - y W_A' (y) = \sum_{\lambda \in \sigma(A), \ \lambda \neq y^{-1}} \frac{\langle \boldsymbol{1}, P_\lambda \boldsymbol{1} \rangle }{(1- \lambda y)^2} (1-y \lambda) \\
= \sum_{\lambda \in \sigma(A), \ \lambda \neq y^{-1}} \frac{\langle \boldsymbol{1}, P_\lambda \boldsymbol{1} \rangle }{1- \lambda y} = \langle \boldsymbol{1}, \boldsymbol{u} \rangle 
\end{align*}
combining all relations above, proves that $\boldsymbol{u}$ satisfies the desired conditions, i.e. $|\boldsymbol{u}|^2=\langle \boldsymbol{1}, \boldsymbol{u} \rangle = W_A(y)$ and $\langle \boldsymbol{u} , A \boldsymbol{u} \rangle =0$. 
\\
\\
\emph{Case: } $y \neq  \lambda_{min}(A)^{-1}, \lambda_{max}(A)^{-1} $
\\
\\
If $y \neq  \lambda_{min}(A)^{-1}, \lambda_{max}(A)^{-1} $, then by assumption of $y$ being a minimum of $W_A(x)$ in the interval $x \in \left[ \lambda_{min}(A)^{-1}, \lambda_{max}(A)^{-1} \right]$, and since $W_A(x)$ is smooth inside this interval, we must have $W_A' (y)=0$. We must also have $y^{-1} \notin \sigma(A)$, which means that $(I-y A)^{-1}$ is well-defined. This means that we can choose $\boldsymbol{u}$ to be 
\begin{align} \label{eq:proof main thm lin alg otp upperB choice u non deg}
\boldsymbol{u} = (I-y A)^{-1} \boldsymbol{1} = \sum_{i=1} \frac{\langle \boldsymbol{u}_i , \boldsymbol{1} \rangle}{1-\lambda_i y} \boldsymbol{u}_i
\end{align}
where in case $A$ is real symmetric, the eigenvectors $\boldsymbol{u}_i$ have been chosen to be real. The condition $W_A' (y)=0$ now becomes 
\begin{align*}
0 =\sum_{i=1} \lambda_i \frac{|\langle \boldsymbol{u}_i , \boldsymbol{1} \rangle|^2}{(1-\lambda_i y)^2} = \langle  (I-y A)^{-1}\boldsymbol{1}, A (I-y A)^{-1} \boldsymbol{1} \rangle = \langle \boldsymbol{u}, A \boldsymbol{u} \rangle
\end{align*}
Thus, $\boldsymbol{u}$ indeed satisfies $\langle \boldsymbol{u}, A \boldsymbol{u} \rangle = 0$. It also follows from the form of $\boldsymbol{u}$ given in (\ref{eq:proof main thm lin alg otp upperB choice u non deg}) that 
\begin{align*}
\langle \boldsymbol{1} , \boldsymbol{u} \rangle = \langle \boldsymbol{1}, (I-y A) (I-y A)^{-2} \boldsymbol{1} \rangle = |\boldsymbol{u}|^2- y \langle \boldsymbol{u}, A \boldsymbol{u} \rangle = |\boldsymbol{u}|^2
\end{align*}
Thus, $\boldsymbol{u}$ also satisfies $|\boldsymbol{u}|^2 = \langle \boldsymbol{1} , \boldsymbol{u} \rangle$, and lastly, by definition of $W_A(y)$, we have $|\boldsymbol{u}|^2 = \langle \boldsymbol{1} , \boldsymbol{u} \rangle = W_A(y)$, proving the proposition in this case. 
\end{proof}

\bigskip

We are now finally in a position to prove Proposition \ref{lemma:linear alg optimization}

\begin{proof}[Proof of Proposition \ref{lemma:linear alg optimization}] \label{eq:lin alg lemma proof} \hfill

It follows straightforwardly from Proposition \ref{prop:existence of optimizer vector} that for any weighted real symmetric adjacency matrix $A$ of $G$, we have 
\begin{align*} 
\sup_{\boldsymbol{v} \in \mathcal{S}_n} \left\{ |\boldsymbol{v}|^2 \ \middle| \ \langle \boldsymbol{v} , A \boldsymbol{v} \rangle =0  \right\}  \geq \min \left\{ W_A(x) \ \middle| \ \lambda_{min}(A)^{-1} \leq x \leq \lambda_{max}(A)^{-1} \right\} 
\end{align*} 
The converse inequality is the content of Proposition \ref{prop:lin alg opt lower bound}, given the assumptions that $G$ is of order $n \geq 3$ and $A \neq 0$, thus proving the proposition. 
\end{proof}

\bigskip

We are now also finally ready to prove Lemma \ref{lemma:2nd equality}. 

\begin{proof}[Proof of Lemma \ref{lemma:2nd equality}] \hfill

It follows from Proposition \ref{lemma:linear alg optimization} that if $G$ is of order $n \geq 3$ and $A \neq 0$, we have
\begin{align} \label{eq:proof main lin alg opt real}
\sup_{\boldsymbol{v} \in \mathcal{S}_n} \left\{ |\boldsymbol{v}|^2 \ \middle| \ \langle \boldsymbol{v} , A \boldsymbol{v} \rangle =0  \right\} = \min \left\{ W_A(x) \ \middle| \ \lambda_{min}(A)^{-1} \leq x \leq \lambda_{max}(A)^{-1} \right\} \ \: ,
\end{align}
where $A$ can be any real symmetric weighted adjacency matrix for $G$. It is also straightforward to verify, using e.g. Cauchy-Schwarz for the LHS, that if $A=0$, then both sides of eq. (\ref{eq:proof main lin alg opt real}) equal $n$, given that $\lambda_{min}(A)^{-1}$ and $\lambda_{max}(A)^{-1}$ are interpreted as $-\infty$ and $+\infty$ (or just demarcating any interval), hence the equality still holds in case $A=0$. Thus, taking an infimum over $A$ on both sides of eq. (\ref{eq:proof main lin alg opt real}) and using the definitions of $\alpha_{\mathcal{S}}(G)$ from definition \ref{def:spherical ind number} gives us that for any simple graph $G$ of order $n \geq 3$  
\begin{align} \label{eq:proof main lin alg opt real inf A}
\alpha_{\mathcal{S}}(G) = \inf_A \min \left\{ W_A(x) \ \middle| \ \lambda_{min}(A)^{-1} \leq x \leq \lambda_{max}(A)^{-1} \right\} \ \: ,
\end{align}
where the infimum is taken over all real symmetric weighted adjacency matrices.

What is left is now to prove that (\ref{eq:proof main lin alg opt real inf A}) still holds if $G$ is of order $n=1$ or $n=2$. If $G$ is empty, then the only possible weighted adjacency matrices are $A=0$, and we have already proven that (\ref{eq:proof main lin alg opt real}) holds if $A=0$, which means that (\ref{eq:proof main lin alg opt real inf A}) still holds if $G$ is empty. The only non-empty simple graph $G$ of order $n \leq 2$ is the complete graph $K_2$ on two vertices. If it is indeed the case that $G=K_2$, then it is not difficult to check directly that
\begin{align*}
1=\alpha_{\mathcal{S}}(K_2) = \inf_A \min \left\{ W_A(x) \ \middle| \ \lambda_{min}(A)^{-1} \leq x \leq \lambda_{max}(A)^{-1} \right\} \ \: .
\end{align*}
We conclude that the equality (\ref{eq:proof main lin alg opt real inf A}) holds for any simple graph $G$, thus proving Lemma \ref{lemma:2nd equality}.
\end{proof}

\section{Concluding Remarks and outlook}

We have here presented two new characterizations of the Lovasz theta function $\vartheta(G)$, one of them showing it to equal the spherical independence number $\alpha_{\mathcal{S}}(G)$ and the other showing $\vartheta(G)$ to be expressible in terms of minimums of weighted walk-generating functions as
$$\vartheta(G)=\inf_{A} \min_{\lambda_{min}(A)^{-1} \leq x \leq \lambda_{max}(A)^{-1}} W_A (x) \ . $$
This was the content Theorem \ref{theorem:char lovasz numb}. This result shows a connection between two very different quantities in algebraic graph theory -- walk-generating functions and the Lovasz theta function, and it also shows $\vartheta(G)$ to equal a very natural relaxation of the independence number -- namely $\alpha_{\mathcal{S}}(G)$ where Boolean indicator vectors $\{0,1\}^n$ are relaxed to a spherical superset of vectors $\mathcal{S}_n$. Meanwhile, the new characterization of $\vartheta(G)$ also has more direct applications and consequences, specifically related to generalizations of the Hoffman bound and the maximum independent set problem. 

In Section \ref{sec:main results} we saw how the new characterization straightforwardly entails natural generalizations of the Hoffman upper bound on $\vartheta(G)$ to non-regular graphs. In particular, the upper bound $\vartheta(G) \leq \min_{\lambda_n^{-1} \leq x \leq 0} W(x)$ was derived in Corollary \ref{cor:precise compl generaliz}, and in Proposition \ref{prop:better Hbound} it was shown that this bound is always better or equal to the previously established generalization of the Hoffman bound on $\vartheta(G)$, $n \left(1 - \delta / \mu_1 \right)$.

An area for further investigation is offered by looking into whether the vector $\boldsymbol{v}$ associated with the spherical independence number can be a tool for tackling the maximum independent set problem. As mentioned in the Introduction, calculating, or even reasonably approximating, $\alpha(G)$ is an NP-hard problem. Meanwhile, $\vartheta(G)$ can be approximated in polynomial time. In light of Theorem \ref{theorem:char lovasz numb}, this means that we in polynomial time can approximate a vector $\boldsymbol{v}$, which satisfies $|\boldsymbol{v}|^2=\alpha_{\mathcal{S}}(G)=\vartheta(G)$, $\langle \boldsymbol{v}, A \boldsymbol{v} \rangle=0$ and $|\boldsymbol{v}|^2=\langle \boldsymbol{1}, \boldsymbol{v} \rangle $, where $A$ is the weighted adjacency matrix minimizing $\lambda_{max}(J-A)$ (i.e. for which $\lambda_{max}(J-A) = \vartheta(G)$). This $\boldsymbol{v}$ equals $(I-y A)^{-1} |_{\supp (I-y A)} \boldsymbol{1}$ plus a potential term in the kernel of $I-yA$ (whose magnitude is fixed by the constraints), and $y$ is the unique minimum of the strictly convex function $W_A (x)$ in the interval $[\lambda_{min}(A)^{-1}, \lambda_{max}(G)^{-1}]$, see Proposition \ref{prop:existence of optimizer vector} for more details. This $\boldsymbol{v}$ can be viewed as a generalization of the indicator function of the maximum independent set of $G$, and it would then obviously be very interesting to see whether $\boldsymbol{v}$ can tell us anything about $\alpha(G)$ or the maximum independent sets. It can be hoped that $\boldsymbol{v}$ provides a sufficiently good first guess for the indicator function of the maximum independent set of $G$ to be of practical interest.

We are of course also left the the obvious open question of whether the new characterizations of $\vartheta(G)$ from Theorem \ref{theorem:char lovasz numb} has any further applications. The same could be asked for the analysis result Theorem \ref{lemma:app:main lemma for reciprocal functs} from Appendix \ref{app:special reciprocal function}.

\section*{Acknowledgments}

The author wants to thank David Earl Roberson, Jereon Zuiddam, Laura Man\v{c}inska, for very helpful comments and discussions.

This publication was supported by VILLUM FONDEN via the QMATH Centre of Excellence (Grant No.10059).

\bibliographystyle{acm}
\bibliography{refs}

\newpage

\begin{appendices}

\section{Duality result for the critical points of special reciprocal functions} \label{app:special reciprocal function}

This section will concern functions that can be written as finite sums of the reciprocals of certain linear functions. These are real functions, $f$, of the type: 
\begin{align} \label{eq:app:first def of reciproc functions}
f(x) = \sum_{i=1}^N \frac{\alpha_i}{1-\beta_i x}
\end{align}
which is defined everywhere on $\mathbb{R} \backslash \bigcup_{i=1}^N \{ \beta_i^{-1} \}$. We only consider finite sums, $N \in \mathbb{N}$, and restrict ourselves to the case where all $\alpha_i$ are nonnegative $\alpha_i \geq 0$. Note that this class of functions generalizes the weighted walk-generating functions of graphs, where weighted here means weighted by any Hermitian adjacency matrix. The new class of functions defined here is in fact more general than the class of weighted walk-generating functions, since the functions of the form \eqref{eq:app:first def of reciproc functions} with all $\alpha_i$'s positive, for example capture any function $g(x)$ that can be written as $g(x)=\langle \boldsymbol{u}, (I-x H)^{-1} \boldsymbol{u} \rangle$, where $\boldsymbol{u}$ is an arbitrary vector, and $H$ is an arbitrary Hermitian matrix.

We shall prove multiple properties of such functions, $f$, that can be written on the form from \cref{eq:app:first def of reciproc functions}, mainly concerning its critical points, i.e. points $x \in \mathbb{R}$ in the domain of $f$ for which its derivative vanish, i.e. $f ' (x)=0$. We are particularly interested in \emph{maximal critical points} which are critical points $x \in \mathbb{R}$ of $f$ for which $f(x)$ equals the maximal value $f$ attains at one of its critical points, i.e., $x$ is a maximal critical point if $f' (x)=0$ and $f(x) \geq f(y)$ for all $y \in \mathbb{R}$ satisfying $f' (y)=0$.

The focus of this section will be on proving the following theorem, which puts a lot of restrictions on the location of maximal critical points of function $f$ of the form \eqref{eq:app:first def of reciproc functions}.
\begin{theorem} \label{lemma:app:main lemma for reciprocal functs}
Let the real function $f : \mathbb{R} \backslash \bigcup_{i=1}^N \{ \beta_i^{-1} \} \rightarrow \mathbb{R}$ be given by 
\begin{align*}
f(x) = \sum_{i=1}^N \frac{\alpha_i}{1-\beta_i x}
\end{align*}
with $\alpha_i \geq 0$ for all $i$. Suppose that $f$ is not a constant function and has at least one critical point. 

Then there must exist an $i$ s.t. $\beta_i > 0$, and a $j$ s.t. $\beta_j < 0$. Suppose now that $x_m$ is a maximal critical point of $f$, meaning that it equals
\begin{align*}
f(x_m)=\max \left\{f(y) \  \big\vert \
y \in \mathbb{R} \ , \ f ' (y)=0 \right\}
\end{align*}
and satisfies $f' (x_m)=0$. Then $x_m$ must lie in the following interval
\begin{align*}
x_m \in \left( \beta_{min}^{-1},  \beta_{max}^{-1} \right)
\end{align*}
where $\beta_{min}:=\min_{1 \leq i \leq N} \beta_i$ and $\beta_{max}:=\max_{1 \leq i \leq N} \beta_i$. $x_m$ must further be a local minimum and is the unique critical point in the interval $\left( \beta_{min}^{-1},  \beta_{max}^{-1} \right)$, meaning
\begin{align*}
f(x_m)=\min_{\beta_{min}^{-1} \leq x \leq \beta_{max}^{-1}} f(x)
\end{align*}
\end{theorem}
See e.g. Figures \ref{fig:Golomb graph walk-generating function} and \ref{fig:P17} for visualizations of this result. Before proving Theorem \ref{lemma:app:main lemma for reciprocal functs}, we shall first prove some intermediate results. We start by making the following definitions to be able to discuss the given class of functions in a more systematic and easily readable manner.

Note first that no generality is lost if we assume that the numbers $\alpha_i$ are all positive and the numbers $\beta_i$ are all different when introducing the class of functions from \cref{eq:app:first def of reciproc functions}. This motivates the following definition of the function $R_{(\boldsymbol{\alpha},\boldsymbol{\beta})}$:
\\
\\
Let $\boldsymbol{\alpha}=\{\alpha_1,...,\alpha_N\}$ and $\boldsymbol{\beta}=\{\beta_1,...,\beta_N\}$ be sets of $k$ real numbers such that all $\alpha_i$ are positive and all $\beta_i$ are distinct, i.e., $\alpha_i > 0$, $\beta_i \neq \beta_j$ for all $i \neq j$. We shall w.l.o.g. always assume the elements in $\boldsymbol{\beta}$ to be ordered in descending order $\beta_1 > \beta_2 > ... > \beta_N$. We now define the function $R_{(\boldsymbol{\alpha},\boldsymbol{\beta})} : \mathbb{R} \ \backslash \ \bigcup_{i=1}^N \{ \beta_i^{-1} \mid \beta_i \neq 0 \} \rightarrow \mathbb{R}$
\begin{align*}
R_{(\boldsymbol{\alpha},\boldsymbol{\beta})}(x) 
:= \sum_{i=1}^N \frac{\alpha_i}{1-\beta_i x}
\end{align*}
We refer to the number $N$ as the \emph{order} of $R_{(\boldsymbol{\alpha},\boldsymbol{\beta})}$. Sometimes we shall denote the order of $f$ as $N(f)$ when we want to specify the function in question. Suppose the function $f$ is of the form described above, i.e., $f=R_{(\boldsymbol{\alpha},\boldsymbol{\beta})}$ for some $(\boldsymbol{\alpha},\boldsymbol{\beta})$. We then define $B^{-}(f)$ and $B^+(f)$ as the sets of respectively negative and positive elements of $\boldsymbol{\beta}$, and let $b^-(f)$ and $b^+(f)$ denote the number of negative and positive elements of $\boldsymbol{\beta}$, i.e., cardinality of $B^+(f)$ and $B^-(f)$
\begin{align*}
B^-(f):= & \{ \beta_i \in \boldsymbol{\beta} \mid \beta_i < 0\} \\
B^+(f):= & \{ \beta_i \in \boldsymbol{\beta} \mid \beta_i > 0\} \\
b^-(f):= & |B^-(f)| \\
b^+(f) := & |B^+(f)|
\end{align*}
Sometimes we shall omit the explicit reference to $f$ when it is clear what function we are talking about. 
Note that since the $\beta$'s are allowed to be $0$, we will not in general have $b^+ + b^- =N$. We are now ready to state our first results concerning functions of the type $R_{(\boldsymbol{\alpha},\boldsymbol{\beta})}$.

\begin{prop}[Smoothness of $R_{(\boldsymbol{\alpha},\boldsymbol{\beta})}$] \label{lemma:app:smoothness of R_ab}
$R_{(\boldsymbol{\alpha},\boldsymbol{\beta})}$ is a smooth function everywhere it is defined, i.e., everywhere on $\mathbb{R} \ \backslash \ \bigcup_{i=1}^N \{ \beta_i^{-1} \mid \beta_i \neq 0 \}$.
\end{prop}
\begin{proof}
This follows immediately from $R_{(\boldsymbol{\alpha},\boldsymbol{\beta})}$ being defined as a finite sum of functions that are all smooth everywhere on $\mathbb{R} \ \backslash \ \bigcup_{i=1}^N \{ \beta_i^{-1} \mid \beta_i \neq 0 \}$.
\end{proof}

\begin{prop}[asymptotic behavior of $R_{(\boldsymbol{\alpha},\boldsymbol{\beta})}$] \label{lemma:app:asymptotics of R_ab}
$R_{(\boldsymbol{\alpha},\boldsymbol{\beta})}$ has the following asymptotic behavior:
\begin{enumerate}
    \item $R_{(\boldsymbol{\alpha},\boldsymbol{\beta})}(x)$ converges to
    $\left\{ \begin{matrix}
        \alpha_y & \ \text{if} \quad 0 \in \boldsymbol{\beta} \ \text{s.t.} \ \beta_y=0  \\
        0 &  \ \text{otherwise}
    \end{matrix} \right.$ in the limits $x \rightarrow \pm \infty$, and $R_{(\boldsymbol{\alpha},\boldsymbol{\beta})}$ is smooth in the intervals $\left( -\infty,\beta_l^{-1} \right)$ and $\left( \beta_r^{-1},+\infty \right)$ where $\beta_l$ is the largest element in $B^-$ and $\beta_r$ is the smallest element in $B^+$, supposing $B^- \neq \emptyset$ and $B^+ \neq \emptyset$ respectively. $R_{(\boldsymbol{\alpha},\boldsymbol{\beta})}$ must also approach $- \infty$ as $x$ approaches $\beta_l^{-1}$ from below and $\beta_r^{-1}$ from above.
    \item Supposing $\beta_i, \beta_{i+1} \in B^{-}$, then $R_{(\boldsymbol{\alpha},\boldsymbol{\beta})}(x)$ will be smooth in the interval $x \in \left( \beta_i^{-1}, \beta_{i+1}^{-1} \right)$ wherein $R_{(\boldsymbol{\alpha},\boldsymbol{\beta})}(x)$ approaches $+\infty$ as $x$ approaches $\beta_i^{-1}$ from above and $R_{(\boldsymbol{\alpha},\boldsymbol{\beta})}(x)$ approaches $-\infty$ as $x$ approaches $\beta_{i+1}^{-1}$ from below. 
    \item Supposing $\beta_i, \beta_{i+1} \in B^{+}$, then $R_{(\boldsymbol{\alpha},\boldsymbol{\beta})}(x)$ will be smooth in the interval $x \in \left( \beta_i^{-1}, \beta_{i+1}^{-1} \right)$ wherein $R_{(\boldsymbol{\alpha},\boldsymbol{\beta})}(x)$ approaches $-\infty$ as $x$ approaches $\beta_i^{-1}$ from above and $R_{(\boldsymbol{\alpha},\boldsymbol{\beta})}(x)$ approaches $+\infty$ as $x$ approaches $\beta_{i+1}^{-1}$ from below. 
    \item Supposing $\beta_N \in B^-$ and $\beta_1 \in B^+$, i.e., $B^-$ and $B^+$ are both non-empty, $R_{(\boldsymbol{\alpha},\boldsymbol{\beta})}(x)$ will be smooth in the interval $\left(\beta_N^{-1}, \beta_1^{-1} \right)$ wherein $R_{(\boldsymbol{\alpha},\boldsymbol{\beta})}(x)$ approaches $+\infty$ as $x$ approaches $\beta_N^{-1}$ from above and as $x$ approaches $\beta_1^{-1}$ from below. 
\end{enumerate}
\end{prop}
\begin{proof}
If $\beta_x=0$ for some $x$, then we simply end up with the constant term $\alpha_x$ in the limits $x \rightarrow \pm \infty$, since all non-constant terms are proportional to $1/x$ and vanish in these limits. 

The smoothness properties of $R_{(\boldsymbol{\alpha},\boldsymbol{\beta})}$ can be derived from Proposition \ref{lemma:app:smoothness of R_ab} by checking that $R_{(\boldsymbol{\alpha},\boldsymbol{\beta})}$ is defined everywhere in the relevant intervals. 

All the asymptotic relations can easily be checked by noting that if we assume $\beta_i > 0$ ($\beta_i < 0$) then the term $\frac{\alpha_i}{1-\beta_i x}$ approaches $+\infty$ ($-\infty$) as $x$ approaches $\beta_i^{-1}$ from below, and $\frac{\alpha_i}{1-\beta_i x}$ will approach $- \infty$ ($+\infty$) as $x$ approaches $\beta_i^{-1}$ from above. This can be seen by setting $x=\beta_i^{-1}+\epsilon$ giving $\frac{\alpha_i}{1-\beta_i x}=-\frac{\alpha_i}{\beta_i} \frac{1}{\epsilon}$, and we by assumption have $\alpha_i > 0$. Since all $\beta$'s are also different, by assumption, no other terms in $R_{(\boldsymbol{\alpha},\boldsymbol{\beta})}$ will diverge in the limit $x \rightarrow \beta_i^{-1}$, and neither will a finite sum of them.
\end{proof}

We next, like in Section \ref{subsec:spherical=walk gen}, make the following definition of the \emph{central strip}, since this is an interval we shall encounter repeatedly in the following arguments. Suppose the function $f$ is of the form $f=R_{(\boldsymbol{\alpha},\boldsymbol{\beta})}$. If $\beta_N <0$ and $\beta_1 > 0$, i.e., if both $B^+(f)$ and $B^-(f)$ are non-empty, we define the \emph{central strip} $S(f)$ of $f$ as the interval 
\begin{align*} 
S(f):=\left( \beta_N^{-1},\beta_1^{-1} \right)
\end{align*}
The following propositions also deal with the critical points of functions of the form $f=R_{(\boldsymbol{\alpha},\boldsymbol{\beta})}$. It is not too hard to show from considering the equation $R_{(\boldsymbol{\alpha},\boldsymbol{\beta})}' (x)=0$, which can be reformulated as polynomial equation of degree at most $2 N(f)-2$, that such a function $f$ can at most have $2 (N(f)-1)$ critical points, assuming $f$ is non-constant. 

\begin{prop}[Properties of the central strip] \label{lemma:app:properties of central strip}
Suppose the function $f$ is of the form $f=R_{(\boldsymbol{\alpha},\boldsymbol{\beta})}$ and the central strip $S(f)$ of $f$ exists. 
\begin{enumerate}
    \item $f$ is smooth everywhere on $S(f)$
    \item $f(x) > 0$ for all $x \in S(f)$
    \item $f''(x) > 0$ for all $x \in S(f)$
    \item $f$ has one and only one critical point in $S(f)$, which is also a minimum on $S(f)$
\end{enumerate}
\end{prop}

\begin{proof}
\emph{1.} follows directly from Proposition \ref{lemma:app:asymptotics of R_ab} (\emph{4.}). 
\\
\emph{2.} can be seen by checking that all terms in $
R_{(\boldsymbol{\alpha},\boldsymbol{\beta})}(x) 
= \sum_{i=1}^N \frac{\alpha_i}{1-\beta_i x}
$ are positive if $x \in S(f)$. 
\\
\emph{3.} can similarly by seen by checking that all terms in $
R_{(\boldsymbol{\alpha},\boldsymbol{\beta})} '' (x) 
=2 \sum_{i=1}^N \beta_i^2 \frac{\alpha_i}{(1-\beta_i x)^3}
$ are non-negative if $x \in S(f)$ and some must be strictly positive since $S(f)$ exists.
\\
\emph{4.} Follows from $f'(x)$ being a monotonically increasing function on $S(f)$, since $f''(x) > 0$ for $x \in S(f)$. It can also be checked that $f'(x)$ grown from $-\infty$ to $+\infty$ in $S(f)$. Thus, $f'(x)$ must equal $0$ for one and only one $x \in S(f)$. 
\end{proof}

\begin{prop} \label{lemma:app:reciprocal:conditions for critical points}
$R_{(\boldsymbol{\alpha},\boldsymbol{\beta})}$ has a critical point, i.e., a point where its derivative vanishes, if and only if we either have $B^- \neq \emptyset$, $B^+ \neq \emptyset$ or we have $\boldsymbol{\beta}=\{0\}$ i.e. $R_{(\boldsymbol{\alpha},\boldsymbol{\beta})}$ is a constant function. 
\end{prop}

\begin{proof}
\emph{$R_{(\boldsymbol{\alpha},\boldsymbol{\beta})}$ has a critical point $\Rightarrow$ $B^- \neq \emptyset$ and $B^+ \neq \emptyset$ or $\boldsymbol{\beta}=\{0\}$}: 
\\
\\
Suppose $R_{(\boldsymbol{\alpha},\boldsymbol{\beta})}$ has a critical point. Consider now the derivative $R_{(\boldsymbol{\alpha},\boldsymbol{\beta})} '(x)$ of $R_{(\boldsymbol{\alpha},\boldsymbol{\beta})}$
\begin{align*}
R_{(\boldsymbol{\alpha},\boldsymbol{\beta})} ' (x) 
= \sum_{i=1}^N \beta_i \frac{\alpha_i}{(1-\beta_i x)^2}
\end{align*}
Since we have $\frac{\alpha_i}{(1-\beta_i x)^2}>0$ everywhere where $R_{(\boldsymbol{\alpha},\boldsymbol{\beta})}$ is defined, we see that $B^{-}=\emptyset$, $B^+ \neq \emptyset$ will lead to $R_{(\boldsymbol{\alpha},\boldsymbol{\beta})} '(x)>0$, and $B^{-} \neq \emptyset$, $B^+ = \emptyset$ will lead to $R_{(\boldsymbol{\alpha},\boldsymbol{\beta})} '(x) < 0$ for all $x$ where $R_{(\boldsymbol{\alpha},\boldsymbol{\beta})}$ is defined. Both cases contradicts the assumption that $R_{(\boldsymbol{\alpha},\boldsymbol{\beta})}$ has a critical point, so we must have either $B^{-} \neq \emptyset$, $B^+ \neq \emptyset$ or we must have $B^{-}= B^{+}= \emptyset$, in which case $\boldsymbol{\beta}=\{0\}$.
\\
\\
\emph{$B^- \neq \emptyset$ and $B^+ \neq \emptyset$ or $\boldsymbol{\beta}=\{0\}$ $\Rightarrow$ $R_{(\boldsymbol{\alpha},\boldsymbol{\beta})}$ has a critical point}:
\\
\\
If $B^- \neq \emptyset$ and $B^+ \neq \emptyset$ then the central strip of $R_{(\boldsymbol{\alpha},\boldsymbol{\beta})}$ exist which, by Proposition $\ref{lemma:app:properties of central strip}$ (\emph{4.}), means that $R_{(\boldsymbol{\alpha},\boldsymbol{\beta})}$ has at least one critical point. If $\boldsymbol{\beta}=\{0\}$, then $R_{(\boldsymbol{\alpha},\boldsymbol{\beta})}$ is a constant function, which means that all points in $\mathbb{R}$ are critical points of $R_{(\boldsymbol{\alpha},\boldsymbol{\beta})}$.
\end{proof}

We are now ready to prove Theorem \ref{lemma:app:main lemma for reciprocal functs}. 

\begin{proof}[Proof of Theorem \ref{lemma:app:main lemma for reciprocal functs}] 
Let $f$ be the function from Theorem \ref{lemma:app:main lemma for reciprocal functs}, i.e. $f : \mathbb{R} \rightarrow \mathbb{R} \backslash \bigcup_{i=1}^N \beta_i^{-1}$ is given by 
\begin{align*} 
f(x) = \sum_{i=1}^N \frac{\alpha_i}{1-\beta_i x}
\end{align*}
with $\alpha_i \geq 0$ for all $i$, and suppose that $f$ is not a constant function and has at least one critical point. 

We can w.l.o.g. assume $f$ to equal some function of the form $R_{(\boldsymbol{\alpha},\boldsymbol{\beta})}$ from the definition in (\ref{eq:app:first def of reciproc functions}), i.e., where $\alpha_i > 0$ for all $i$, and $\boldsymbol{\beta}$ is strictly ordered as $\beta_i < \beta_j$ for $i < j$. Getting an arbitrary $f$ on this form can be achieved by deleting zero-terms, combining terms or relabeling dummy-variables. Since $f$ by assumption is non-constant and has at least one critical point, it follows directly from Proposition \ref{lemma:app:reciprocal:conditions for critical points} that $B^+ (f) \neq \emptyset$ and $B^- (f) \neq \emptyset$, meaning that $\beta_i < 0$ and $\beta_j > 0$ for some $i$ and $j$. This proves the first result from Theorem \ref{lemma:app:main lemma for reciprocal functs}, and shows that $\beta_1 > 0$ and $\beta_N < 0$, entailing that the central strip $S(f)$ of $f$ exists. It follow directly from Proposition \ref{lemma:app:properties of central strip} that if a maximal critical point $x_m$ of $f$ lies in the central strip $S(f)$, then $x_m$ must be the unique critical point in $S(f)$, which is also a minimum of $f$ in $S(f)$. Thus, the rest of Theorem \ref{lemma:app:main lemma for reciprocal functs} follows if we prove the following entailment
\begin{align} \label{eq:app:proof of main reciprocal thm step 2}
f'(x_m)=0, \quad f(x_m)=\max_y \{f(y) \mid f ' (y)=0\} \quad \Rightarrow \quad x_m \in S(f)
\end{align}
We shall now prove the statement (\ref{eq:app:proof of main reciprocal thm step 2}) above. Note that we can start by assuming w.l.o.g. that $\beta_i \neq 0$ for all $i$ when proving (\ref{eq:app:proof of main reciprocal thm step 2}), since adding or subtracting a constant term from $f$ does not change the location of its maximal critical points, nor does it change $S(f)$. We now prove the statement (\ref{eq:app:proof of main reciprocal thm step 2}) by induction on the order $N \in \mathbb{N}$ of $f$. Remember, the order $N(R_{(\boldsymbol{\alpha},\boldsymbol{\beta})})$ of $R_{(\boldsymbol{\alpha},\boldsymbol{\beta})}$ is the number of terms in the sum defining $R_{(\boldsymbol{\alpha},\boldsymbol{\beta})}$, i.e., $N(R_{(\boldsymbol{\alpha},\boldsymbol{\beta})})=|\boldsymbol{\beta}|$. 

We shall specifically start by proving that if $g$ is a non-constant function of the form $g=R_{(\boldsymbol{\alpha},\boldsymbol{\beta})}$, with $\beta_i \neq 0$ for all $i$, and with at least one critical point, and has order $N(g)=2$, then any maximal critical point of $g$ must lie in $S(g)$. This is our base-case. The reason why $N=2$ has to be the base-case is that Proposition \ref{lemma:app:reciprocal:conditions for critical points} implies $N(f) \geq b^+(f) + b^-(f) \geq 1 + 1 \geq 2$ for $f$ non-constant and containing critical points. In the inductive step, we shall prove that if for all non-constant functions $g$ of the form $g=R_{(\boldsymbol{\alpha},\boldsymbol{\beta})}$ with at least one critical point, with $\beta_i \neq 0$ for all $i$, and with order $N(g) \leq k$, any maximal critical points of $g$ lies in $S(g)$, then the same will be true for any such function of order $N=k+1$. This will prove the statement (\ref{eq:app:proof of main reciprocal thm step 2}) for any relevant function $f$, since the order $N(f)$ must be a natural number by definition, thus completing the proof of Theorem \ref{lemma:app:main lemma for reciprocal functs}.
\\
\\
\emph{Proof of the base-case}: Let $g$ be any non-constant function of the form $g=R_{(\boldsymbol{\alpha},\boldsymbol{\beta})}$ with at least one critical point, with $\beta_i \neq 0$ for all $i$, and of order $N(g)=2$. Then by Proposition \ref{lemma:app:reciprocal:conditions for critical points}, we must have $\beta_1 > 0$ and $\beta_2 < 0$, which means $S(g)=\left( - |\beta_2|^{-1}, \beta_1^{-1} \right)$ and 
\begin{align*}
g(x)=\frac{\alpha_1}{1-\beta_1 x}+\frac{\alpha_2}{1+|\beta_2| x}
\end{align*}
We now see that for any $x,y$ in the domain of $g$, if $x \in S(g)$ and $y \notin S(g)$, we must have $g(x) \geq g(y)$. This is true since for $x \in S(f)$ both terms $\frac{\alpha_1}{1-\beta_1 x}$ and $\frac{\alpha_2}{1+|\beta_2| x}$ are positive. If $y < - |\beta_2|^{-1}$, then we straightforwardly get $\frac{\alpha_2}{1+|\beta_2| y} < 0$ and $\frac{\alpha_1}{1-\beta_1 y} < \frac{\alpha_1}{1-\beta_1 x}$ follows since $y  < \beta_2^{-1} < x < \beta_1^{-1}$, thus $g(y) < g(x)$. Similarly, if $y > \beta_1^{-1}$, a similar argument with the roles of the two terms reversed also shows $g(y) < g(x)$. Since $g$ by Proposition \ref{lemma:app:reciprocal:conditions for critical points} must have a single critical point $x_c$ in $S(g)$, this critical point must be the only maximal critical point of $g$, since $g(x_c) > g(y_c)$ holds for any critical point $y_c \notin S(g)$ by the result just proven (we do not even have to use that $y_c$ is a critical point, just that $y_c \notin S(g)$).  
\\
\\
\emph{Proof of the inductive step}:
let $k \in \mathbb{N}$ with $k \geq 2$. Assume now that for all non-constant functions $g$ of the form $g=R_{(\boldsymbol{\alpha},\boldsymbol{\beta})}$ with at least one critical point, with $\beta_i \neq 0$ for all $i$, and of order $N(g) \leq k$, any maximal critical points of $g$ lies in $S(g)$. Let now $f$ be any non-constant function of the form $f=R_{(\boldsymbol{\alpha},\boldsymbol{\beta})}$ with at least one critical point, with $\beta_i \neq 0$ for all $i$, and of order $N(f) = k+1$. We want to prove that any maximal critical point $x_m$ of $f$ lies in $S(f)$. This will be proven by contradiction. Suppose that $x_m$ is a maximal critical point of $f$ satisfying $x_m \notin S(f)$. Then we must either have $x_m < \beta_{k+1}^{-1}$ or $x_m > \beta_1^{-1}$ (Note that since $N(f)=k+1$, we have $\beta_{min}=\beta_{k+1}$). We shall here focus on the case $x_m < \beta_{k+1}^{-1}$. The argument leading to contradiction in the case $x_m > \beta_1^{-1}$ is in essence identical to the case $x_m < \beta_{k+1}^{-1}$. The only differences being some sign-changes and label-changes.

Consider the case $x_m < \beta_{k+1}^{-1}$. To prove that this is a contradiction, we introduce the function $f|_k$, defined as 
\begin{align*}
f|_k (x):= R_{(\boldsymbol{\alpha} \backslash \alpha_{k+1},\boldsymbol{\beta} \backslash \beta_{k+1})}=\sum_{i=1}^k \frac{\alpha_i}{1-\beta_i x}=f(x)-\frac{\alpha_{k+1}}{1-\beta_{k+1} x}
\end{align*}
Note first that $x_m$ being a maximal critical point entails $b^{-}(f) \geq 2$, since if we had $b^{-} (f)=1$ then $\beta_i > 0$ for all $1 \leq i \leq k$, leading to $f|_k (x_m) < f|_k (y)$ for any $y \in S(f)$ and we also get $\frac{\alpha_{k+1}}{1-\beta_{k+1}x_m} <0 < \frac{\alpha_{k+1}}{1-\beta_{k+1} y }$, thus entailing $f(y) > f(x_m)$ for any $y \in S(f)$. Since the central strip $S(f)$ of $f$ contains a critical point by Proposition \ref{lemma:app:reciprocal:conditions for critical points}, this contradicts the assumption of $x_m$ being a maximal critical point. Thus, we must have $b^{-}(f) \geq 2$, entailing $b^{-} (f|_k ) \geq 1$. Since $b^{+} (f|_k )=b^{+} (f) \geq 1$ we now have $B^{+}(f|_k) , B^{-}(f|_k)  \neq \emptyset$, and by Proposition \ref{lemma:app:reciprocal:conditions for critical points}, $f|_k$ must be a non-constant function which has critical points. By assumption, we have $\beta_i \neq 0$ for all $i$ and $f|_k$ is of the form $f|_k = R_{(\boldsymbol{\alpha} \backslash \alpha_{k+1},\boldsymbol{\beta} \backslash \beta_{k+1})}$. Finally, note that the order of $f|_k$ equals $N(f|_k )=k$. Hence, by assumption of the inductive hypothesis, any maximal critical point of $f|_k$ must lie in the central strip $S(f|_k )$ of $f|_k$. There are now two cases to consider. Either $x_m \in S(f|_k)$ or $x_m \notin S(f|_k)$. We shall prove that both lead to a contradiction.

\emph{Suppose $x_m \in S(f|_k)$}. Note first that since $\frac{\alpha_{k+1}}{1-\beta_{k+1}x_m} <0$, which follows from the assumption $x_m < \beta_{k+1}^{-1}$, we have $f(x_m) < f|_k (x_m)$. Note now that 
\begin{align} \label{eq:app:proof of main reciprocal thm step 3}
f|_k ' (x_m)=f ' (x_m) - \frac{\beta_{k+1} \alpha_{k+1}}{(1-\beta_{k+1}x_m)^2}=0- \frac{\beta_{k+1} \alpha_{k+1}}{(1-\beta_{k+1}x_m)^2} > 0
\end{align}
where $f ' (x_m)=0$ follows from $x_m$ being a critical point. Consider now any $y \in S(f) \subset S(f|_k)$. Then we must have $y > x_m$ by assumption of $x_m < \beta_{k+1}^{-1}$. Since $f|_k$ by Proposition \ref{lemma:app:properties of central strip} is strictly convex in its central strip $S(f|_k )$, we must have $f|_k (x_m) < f|_k (y)$ since $f|_k$ must be monotonically increasing between the two points, which follows from $f|_k ' (x_m) > 0$ and strict convexity. Finally, since $\frac{\alpha_{k+1}}{1-\beta_{k+1}y}  > 0$ by the assumption of $y \in S(f)$, we get $f|_k (y) < f(y)$. Putting all the derived inequalities together gives 
\begin{align} \label{eq:app:proof of main reciprocal thm step 4}
f(x_m) < f|_k (x_m) < f|_k (y) < f(y)
\end{align}
which holds for all $y \in S(f) $. By Proposition \ref{lemma:app:properties of central strip}, $f$ must have a critical point $x_c$ in $S(f)$ and it follows from \cref{eq:app:proof of main reciprocal thm step 4} that $f(x_m) < f(x_c)$, contradicting the assumption of $x_m$ being a maximal critical point of $f$. 

\emph{Suppose now $x_m \notin S(f|_k)$}. We must still by \cref{eq:app:proof of main reciprocal thm step 3} have $f|_k ' (x_m) > 0$ in this case. and it also still follows that $f(x_m) < f|_k (x_m)$. Since $x_m \notin S(f|_k)$, it follows that $x_m < \beta_k^{-1}$ and Proposition \ref{lemma:app:asymptotics of R_ab}(\emph{1.}) or (\emph{2.}) implies that $f|_k ' $ must be smooth in an interval $[x_m,\beta_a^{-1})$ for some $1 < a \leq k$ with $\beta_a < 0$, and $f|_k ' (y)$ must approach $- \infty$ as $y$ approaches $\beta_a^{-1}$. Since we have $f|_k ' (x_m) > 0$, smoothness entails that there must exist some $y \in (x_m,\beta_a^{-1})$ at which $f|_k ' (y) = 0$ and $f|_k (y) > f|_k (x_m)$, since $f|_k$ must be monotonically increasing in the interval $(x_m,y)$. Thus, $y$ is a critical point of $f|_k$. By assumption of the inductive hypothesis, the maximal critical point of $f|_k$ must lie in its central strip, as explained above. Call this maximal critical point $z_m$. Then $z_m \in S(f|_k)$ and $f|_k (z_m) > f|_k (y)$. Combining all inequalities derived so far gives us 
\begin{align}\label{eq:app:proof of main reciprocal thm step 5}
f(x_m) < f|_k (x_m) < f|_k (y) < f|_k (z_m)
\end{align}
By Proposition \ref{lemma:app:properties of central strip}, $z_m$ is the minimum of $f|_k$ in $S(f|_k)$. Thus, since $S(f) \subset S(f|_k)$, we have that for any $x \in S(f)$, $f|_k (x) \geq f|_k (z_m)$, and \cref{eq:app:proof of main reciprocal thm step 5} now entails that $f|_k (x) > f(x_m)$. Since we by assumption of $x \in S(f)$ have $x > \beta_{k+1}^{-1}$, we must have $\frac{\alpha_{k+1}}{1-\beta_{k+1}x} > 0$, giving us $f(x) > f|_k (x)$. Combining this with the last inequality entails $f(x) > f(x_m)$ which holds for any $x \in S(f)$. In particular, it must hold for the critical point of $f$ in $S(f)$, which exist by Proposition \ref{lemma:app:properties of central strip}, contradicting the assumption of $x_m$ being a maximal critical point of $f$. 

We thus see that if $x_m$ is a maximal critical point of $f$, we cannot have $x_m < \beta_{k+1}^{-1}$. By an argument essentially identical to the one just seen, we can neither have $x_m > \beta_1^{-1}$. In this new argument, the only difference is that we define $f|_k(x):=f(x)-\frac{\alpha_1}{1-\beta_1 x}$, and that we now get $f|_k ' (x_m) < 0$ instead of $f|_k ' (x_m) > 0$ and have to use the mirrored asymptotic properties of $f(x)$ for $x > 0$ from Proposition \ref{lemma:app:asymptotics of R_ab} instead of the case $x < 0$. The conclusion is that we cannot have $x_m \notin S(f)$, thus proving the inductive step. Hence, we must have $x_m \in S(f)$, which by induction completes the proof of Theorem \ref{lemma:app:main lemma for reciprocal functs}.

\end{proof}

\section{Proof of sub-multiplicativity for the spherical independence numbers} \label{app:submultiplic}

In this section, we prove that $\alpha_{\mathcal{S}}(G)$ is a sub-multiplicative graph function under the strong graph product $\boxtimes$ defined in Section \ref{sec:preliminaries}.

\begin{lemma} \label{lemma:submultiplicativity}
The spherical independence number $\alpha_{\mathcal{S}}$ is sub-multiplicative under the strong graph product, i.e. for any graphs $G$ and $H$, the following holds
\begin{align*}
\alpha_{\mathcal{S}}(G \boxtimes H) \leq \alpha_{\mathcal{S}} (G) \ \alpha_{\mathcal{S}}(H)
\end{align*} 
\end{lemma}
\begin{proof}
Let $G$ and $H$ be any graphs of order $n_G$ and $n_H$, and let $A_G$ and $A_H$ be any associated real symmetric (hermitian) adjacency matrices. Then, for any real numbers $\gamma_G , \gamma_H \in \mathbb{R}$, the matrix $A_{G,H}(\gamma_G, \gamma_H)$ defined below is a weighted adjacency matrix for $G \boxtimes H$, which can be shown from the definition of $\boxtimes$ given in Section \ref{sec:preliminaries}
\begin{align*}
A_{G,H} (\gamma_G,\gamma_H):=(A_G+\gamma_G I_G) \otimes (A_H+\gamma_H I_H)-\gamma_G \gamma_H I_G \otimes I_H
\end{align*}
where $\otimes$ is the standard tensor-product, which can simply be viewed as the Kronecker product, and $I_G$ and $I_H$ are the identity matrices acting on the same spaces as $I_G$ and $I_H$.

It is easy to verify that $\left\{ \boldsymbol{u}_{j}^G \otimes \boldsymbol{u}_{k}^H \right\}_{j,k=1}^{n_G, n_H}$, where $\boldsymbol{u}_{j}^G$ and $\boldsymbol{u}_{k}^H$ are the ortho-normal eigenvectors of $A_G$ and $A_H$ respectively, is a set of ortho-normal eigenvectors of $A_{G,H} (\gamma_G,\gamma_H)$ which span the entire vector-space, hence consists of all eigenvectors of $A_{G,H} (\gamma_G,\gamma_H)$. It follows that the eigenvalues of $A_{G,H} (\gamma_G,\gamma_H)$ are given by
\begin{align} \label{eq:eigenvalus of G box H}
\lambda_{j k}:=(\lambda_j^G+\gamma_G)(\lambda_{k}^H+\gamma_H)-\gamma_G \gamma_H
\end{align}
where $\lambda_j^G$ and $\lambda_k^H$ ranges over the eigenvalues of $A_G$ and $A_H$ respectively. 
\\
\\
Let now $\lambda_{min}(G)$, $\lambda_{min}(H)$ and $\lambda_{max}(G)$, $\lambda_{max}(H)$ be the smallest, and respectively largest, eigenvalues of $A_G$ and $A_H$. We now restricts our choices for $\gamma_G$ and $\gamma_H$ s.t. the numbers satisfy
\begin{align}
\gamma_G \geq - \lambda_{min}(G) \qquad & \text{or} \qquad \gamma_G \leq - \lambda_{max}(G) \label{eq:choices for submult gammas G} \\ \label{eq:choices for submult gammas H}
\gamma_H \geq - \lambda_{min}(H) \qquad & \text{or} \qquad \gamma_H \leq - \lambda_{max}(H)
\end{align}
These choices of $\gamma_G$ and $\gamma_H$ entails, by \cref{eq:eigenvalus of G box H}, that depending on whether $\gamma_G$ and $\gamma_H$ are positive or negative, the smallest eigenvalue $\lambda_{min}(G \boxtimes H)$ and largest eigenvalue $\lambda_{max}(G \boxtimes H)$ of $A_{G,H} (\gamma_G,\gamma_H)$ will be given by
\begin{align} \label{eq:smallest eigenvalue of G box H}
\lambda_{min}(G \boxtimes H)=(\eta_{\gamma_G}(G)+\gamma_G)(\eta_{\gamma_H}(H)+\gamma_H)-\gamma_G \gamma_H \\ \label{eq:largest eigenvalue of G box H}
\lambda_{max}(G \boxtimes H)=(\mu_{\gamma_G}(G)+\gamma_G)(\mu_{\gamma_H}(H)+\gamma_H)-\gamma_G \gamma_H
\end{align}
where
\begin{align*}
(\eta_{\gamma_G}(G), \eta_{\gamma_H}(H)) = \left\{ \begin{matrix}
(\lambda_{min}(G), \lambda_{min}(H)) \qquad \text{if} \ \: \gamma_G \geq 0, \ \gamma_H \geq 0 \\
(\lambda_{max}(G), \lambda_{min}(H)) \qquad \text{if} \ \: \gamma_G \geq 0, \ \gamma_H \leq 0 \\
(\lambda_{min}(G), \lambda_{max}(H)) \qquad \text{if} \ \: \gamma_G \leq 0, \ \gamma_H \geq 0 \\
(\lambda_{max}(G), \lambda_{max}(H)) \qquad \text{if} \ \: \gamma_G \leq 0, \ \gamma_H \leq 0 \\
\end{matrix} \right. \\
(\mu_{\gamma_G}(G), \mu_{\gamma_H}(H)) = \left\{ \begin{matrix}
(\lambda_{max}(G), \lambda_{max}(H)) \qquad \text{if} \ \: \gamma_G \geq 0, \ \gamma_H \geq 0 \\
(\lambda_{min}(G), \lambda_{max}(H)) \qquad \text{if} \ \: \gamma_G \geq 0, \ \gamma_H \leq 0 \\
(\lambda_{max}(G), \lambda_{min}(H)) \qquad \text{if} \ \: \gamma_G \leq 0, \ \gamma_H \geq 0 \\
(\lambda_{min}(G), \lambda_{min}(H)) \qquad \text{if} \ \: \gamma_G \leq 0, \ \gamma_H \leq 0 \\
\end{matrix} \right.
\end{align*}
Note that we always have $\lambda_{min}(G \boxtimes H) \leq 0$ and $\lambda_{max}(G \boxtimes H) \geq 0$, while $\frac{-1}{\gamma_G \gamma_H}$ can both be smaller and larger than $0$, depending on the choices for $\gamma_G$ and $\gamma_H$ in (\ref{eq:choices for submult gammas G}) and (\ref{eq:choices for submult gammas H}). It can be checked using (\ref{eq:smallest eigenvalue of G box H}) and (\ref{eq:largest eigenvalue of G box H}) and the conditions on $\gamma_G$ and $\gamma_H$ in (\ref{eq:choices for submult gammas G}) and (\ref{eq:choices for submult gammas H}), that in any of the $4$ cases of choices for $\gamma_G$ and $\gamma_H$, we always have
\begin{align} \label{eq:proof submult good range}
\lambda_{min}(G \boxtimes H)^{-1} \leq \frac{-1}{\gamma_G \gamma_H} \leq \lambda_{max}(G \boxtimes H)^{-1}
\end{align}
In case $\lambda_{min}(G \boxtimes H) = 0$ or $\lambda_{max}(G \boxtimes H)=0$, which will only occur if $A_G = 0$ or $A_H = 0$, we simply denote $\lambda_{min}(G \boxtimes H)^{-1} = - \infty$ or $\lambda_{max}(G \boxtimes H)^{-1}=+ \infty$ in accordance with the notation from Theorem \ref{theorem:char lovasz numb}. 

Using (\ref{eq:proof submult good range}) and the characterization of $\alpha_{\mathcal{S}}(G)$ from Theorem \ref{theorem:char lovasz numb}, we now get
\begin{align} \label{eq:der of subadditivity step walk gen func}
\alpha_{\mathcal{S}} (G \boxtimes H) \leq \min_{\lambda_{min}(G \boxtimes H)^{-1} \leq x \leq \lambda_{max}(G \boxtimes H)^{-1}} W_{A_{G,H}(\gamma_G,\gamma_H)}(x) \leq  W_{A_{G,H}(\gamma_G,\gamma_H)} \left( \frac{-1}{\gamma_G \gamma_H} \right) 
\end{align}
in case $A_G$ and $A_H$ were real symmetric weighted adjacency matrices.
\\
\\
Let now $\boldsymbol{1}_G$, $\boldsymbol{1}_H$ and $\boldsymbol{1}_{G,H}$ be the all-ones vectors on the same spaces as $A_G$, $A_H$ and $A_{G,H}(\gamma_G, \gamma_H)$ respectively. Clearly $\boldsymbol{1}_{G,H}=\boldsymbol{1}_G \otimes \boldsymbol{1}_H$. Since the eigenvectors of $A_{G,H}(\gamma_G, \gamma_H)$ are given by $\left\{ \boldsymbol{u}_{j}^G \otimes \boldsymbol{u}_{k}^H \right\}_{j,k=1}^{n_G, n_H}$ with corresponding eigenvalues given by \eqref{eq:eigenvalus of G box H}, as discussed above, we get
\begin{align}
W_{A_{G,H}(\gamma_G,\gamma_H)} \left( \frac{-1}{\gamma_G \gamma_H} \right) = \sum_{j,k = 1}^{n_G, n_H} \frac{\left| \langle \boldsymbol{1}_{G,H}, \boldsymbol{u}_{j}^G \otimes \boldsymbol{u}_{k}^H \rangle  \right|^2}{1-(-1)/(\gamma_G \gamma_H) \left( (\lambda_j^G+\gamma_G)(\lambda_{k}^H+\gamma_H)-\gamma_G \gamma_H \right)} \notag \\
= \sum_{j,k = 1}^{n_G, n_H} \frac{|\langle \boldsymbol{1}_G , \boldsymbol{u}_j^G \rangle |^2 |\langle \boldsymbol{1}_H , \boldsymbol{u}_k^H \rangle |^2 }{(1+\lambda_j^G / \gamma_G) (1+\lambda_k^H / \gamma_H)} = \left( \sum_{j=1}^{n_G} \frac{|\langle \boldsymbol{1}_G , \boldsymbol{u}_j^G \rangle |^2}{1-\lambda_j^G / (-\gamma_G)} \right) \left( \sum_{k=1}^{n_H} \frac{|\langle \boldsymbol{1}_H , \boldsymbol{u}_k^H \rangle |^2}{1-\lambda_k^H / (-\gamma_H)} \right) \notag \\ \label{eq:proof submultiplicativity split step}
= W_{A_G}(-1/\gamma_G) W_{A_H} (-1/\gamma_H)
\end{align}
Hence, combining \eqref{eq:proof submultiplicativity split step} with \eqref{eq:der of subadditivity step walk gen func} and using the assumptions in \eqref{eq:choices for submult gammas G} and \eqref{eq:choices for submult gammas H}, we have that for any $\gamma_G, \gamma_H \in \mathbb{R}$ satisfying $\gamma_G \geq - \lambda_{min}(G)$ or $\gamma_G \leq - \lambda_{max}(G)$, and satisfying $\gamma_H \geq - \lambda_{min}(H)$ or $\gamma_H \leq - \lambda_{max}(H)$, we have
\begin{align*}
\alpha_{\mathcal{S}}(G \boxtimes H) \leq W_{A_G}(-1/\gamma_G) W_{A_H} (-1/\gamma_H)
\end{align*}
Changing variables to $x=-1/\gamma_G$ and $y=-1/\gamma_H$, and minimizing over the allowed intervals gives us
\begin{align*}
\alpha_{\mathcal{S}}(G \boxtimes H) \leq \left( \min_{\lambda_{min}(G)^{-1} \leq x \leq \lambda_{max}(G)^{-1}} W_{A_G}(x) \right) \left(\min_{\lambda_{min}(H)^{-1} \leq y \leq \lambda_{max}(H)^{-1}}   W_{A_H} (y) \right)
\end{align*}
Since $A_G$ and $A_H$ were arbitrary real symmetric (Hermitian) weighted adjacency matrices for $G$ and $H$, we can take an infimum over all $A_G$ and $A_H$ in the inequality above, which finally gives us
\begin{align*}
\alpha_{\mathcal{S}}(G \boxtimes H) \leq \left( \inf_{A_G} \min_{\lambda_{min}(G)^{-1} \leq x \leq \lambda_{max}(G)^{-1}} W_{A_G}(x) \right) \left( \inf_{A_H} \min_{\lambda_{min}(H)^{-1} \leq y \leq \lambda_{max}(H)^{-1}}   W_{A_H} (y) \right) \\
= \alpha_{\mathcal{S}}(G)  \alpha_{\mathcal{S}} (H)
\end{align*}
Thus finally proving sub-multiplicativity for $\alpha_{\mathcal{S}}$. 
\end{proof}

\end{appendices}

\end{document}